\documentclass{article}
\usepackage{xr}
\usepackage{geometry,amsmath,amssymb,enumerate,bbm,color,theorem}
\usepackage[american]{babel}
\newcommand{\assign}{:=}
\newcommand{\longhookrightarrow}{{\lhook\joinrel\relbar\joinrel\rightarrow}}
\newcommand{\nin}{\not\in}
\newcommand{\tmdummy}{$\mbox{}$}
\newcommand{\tmmathbf}[1]{\ensuremath{\boldsymbol{#1}}}
\newcommand{\tmop}[1]{\ensuremath{\operatorname{#1}}}
\newcommand{\tmrsub}[1]{\ensuremath{_{\textrm{#1}}}}
\newcommand{\tmstrong}[1]{\textbf{#1}}
\newcommand{\tmtextit}[1]{{\itshape{#1}}}
\newenvironment{enumeratealpha}{\begin{enumerate}[a{\textup{)}}] }{\end{enumerate}}
\newenvironment{enumeratenumeric}{\begin{enumerate}[1.] }{\end{enumerate}}
\newenvironment{itemizedot}{\begin{itemize} }{\end{itemize}}
\newenvironment{proof}{\noindent\textbf{Proof\ }}{\hspace*{\fill}$\Box$\medskip}
\definecolor{grey}{rgb}{0.75,0.75,0.75}
\definecolor{orange}{rgb}{1.0,0.5,0.5}
\definecolor{brown}{rgb}{0.5,0.25,0.0}
\definecolor{pink}{rgb}{1.0,0.5,0.5}
\newtheorem{corollary}{Corollary}
\newtheorem{definition}{Definition}
\newtheorem{lemma}{Lemma}
\newtheorem{proposition}{Proposition}
{\theorembodyfont{\rmfamily}\newtheorem{remark}{Remark}}
\newtheorem{theorem}{Theorem}

\begin{document}

\title{Higher classes of ordinals induced by $<_1$}\author{Parm{\'e}nides
Garc{\'i}a Cornejo}\maketitle

\begin{abstract}
  In this article we want to see that it is possible to iterate and generalize
  the notions presented in {\cite{GarciaCornejo1}} such
  that we can obtain the higher or thinner classes of ordinals induced by
  the $<_1$-relation as solutions of $\langle \alpha, t \rangle$-conditions.
  Indeed, these results are so general that in a future article we will
  be able to characterize the ordinal $| \Pi_1^1 - \tmop{CA}_0 |$ as the smallest ordinal
  of the thinnest class induced by $<_1$.
\end{abstract}

\section{$\tmop{Class} (n)$}

\begin{definition}
  We define by recursion on $\omega$\\
  $\tmop{Class} (1) \assign \mathbbm{E}$;\\
  $\tmop{Class} (n + 1) \assign \{\alpha \in \tmop{OR} | \alpha \in
  \tmop{Class} (n) \wedge \alpha ( +^n) \in \tmop{Class} (n) \wedge \alpha <_1
  \alpha (+^n)\}$,\\
  where for $\alpha \in \tmop{Class} (n)$ we define \\
  $\alpha (+^n) \assign
  \left\{ \begin{array}{l}
    \min \{\beta \in \tmop{Class} (n) | \alpha < \beta\} \text{ iff }
    \{\beta \in \tmop{Class} (n) | \alpha < \beta\} \neq \emptyset\\
    \infty \text{ otherwise }
  \end{array} \right.$, and we make the conventions $\infty \nin \tmop{OR}$
  and $\forall \gamma \in \tmop{OR} . \gamma < \infty$
\end{definition}

In the upper definition of $(+^n)$, which we can call ``the successor
functional of $\tmop{Class} (n)$'', we needed to consider the case when, for
$\alpha \in \tmop{Class} (n)$, such ``successor of $\alpha$ in $\tmop{Class}
(n)$'' may not exist. We want to tell the reader that this is just a merely
formal necessity: one of our purposes is to show that such successor always
exists and that $\tmop{Class} (n)$ ``behaves like the class $\mathbbm{E}$'' in
the sense of being a closed unbounded class of ordinals. This is one of the
important results we want to generalize, although it's proof will take a lot
of effort.

Let's see now some basic properties of the elements of $\tmop{Class} (n)$.

\begin{proposition}
  \label{Class(n)_first_properties}{\tmdummy}
  
  \begin{enumeratenumeric}
    \item $\forall n, i \in [1, \omega) .i \leqslant n \Longrightarrow
    \tmop{Class} (n) \subset \tmop{Class} (i)$
    
    \item For any $n \in [1, \omega)$ and any $\alpha \in \tmop{Class} (n)$
    define recursively on $[0, n - 1]$\\
    $\alpha_n \assign \alpha$, $\alpha_{n - (k + 1)} \assign \alpha_{n - k}
    (+^{n - (k + 1)})$.\\
    Then $\forall i \in [1, n] . \alpha_i \in \tmop{Class} (i)$ and $\alpha =
    \alpha_n <_1 \alpha_{n - 1} <_1 \ldots <_1 \alpha_2 <_1 \alpha_1 <_1
    \alpha_1 2$.
    
    \item For any $n \in [1, \omega)$ and any $\alpha \in \tmop{Class} (n)$
    consider the sequence defined in 2.\\
    If $\alpha <_1 \alpha_1 2 + 1$ then $\alpha \in \tmop{Lim} \tmop{Class}
    (n)$.
    
    \item $\forall n \in [2, \omega) . \tmop{Class} (n) \subset \tmop{Lim}
    \tmop{Class} (n - 1)$.
    
    \item $\forall n, m \in [1, \omega) \forall \alpha . (m < n \wedge \alpha
    \in \tmop{Class} (n)) \Longrightarrow \alpha \in \tmop{Class} (m) \wedge
    \alpha (+^m) < \alpha (+^n)$.
  \end{enumeratenumeric}
\end{proposition}

\begin{proof}
  {\color{orange} 1, 2 and 5 are left to the reader.}
  
  {\noindent}3.\\
  Let $\rho \in \alpha$ be arbitrary. Define $B_{\rho} \assign \{\rho\} \cup
  \{\alpha_i, \alpha_i 2| n \in \{1, \ldots, n\}\}$. Since $\alpha <_1
  \alpha_1 2 + 1$ and $B_{\rho} \subset_{\tmop{fin}} \alpha_1 2 + 1$, then
  there exists an ($<, <_1, +$)-isomorphism $h_{\rho} : B_{\rho}
  \longrightarrow h [B_{\rho}] \subset \alpha$ with \\
  $h_{\rho} |_{\alpha} = \tmop{Id}_{\alpha}$. Note this implies the following
  facts in the following order (the order is important since the later facts
  use the previous to assert their conclusion):

  (1') \ $\forall i \in [1, n] . \alpha_i <_1 \alpha_i 2 \Longleftrightarrow
  h_{\rho} (\alpha_i) <_1 h_{\rho} (\alpha_i) 2$. So $h_{\rho} (\alpha_i) \in
  \mathbbm{E} = \tmop{Class} (1)$.
  
  (2') \ $\forall i \in [2, n] . \alpha_i <_1 \alpha_1 \Longleftrightarrow
  h_{\rho} (\alpha_i) <_1 h_{\rho} (\alpha_1)$. \\
  \ \ \ \ \ \ \ \ \ \ So by (1') and $<_1$-connectedness $\forall i \in [2,
  n] .h_{\rho} (\alpha_i) \in \tmop{Class} (2)$.
  
  (3') \ $\forall i \in [3, n] . \alpha_i <_1 \alpha_2 \Longleftrightarrow
  h_{\rho} (\alpha_i) <_1 h_{\rho} (\alpha_2)$.\\
  \ \ \ \ \ \ \ \ \ \ So by (1') and (2') and $<_1$-connectedness $\forall i
  \in [3, n - 1] .h_{\rho} (\alpha_i) \in \tmop{Class} (3)$.
  
  ... (inductively)
  
  (n') \ $\alpha_n <_1 \alpha_{n - 1} \Longleftrightarrow h_{\rho} (\alpha_n)
  <_1 h_{\rho} (\alpha_{n - 1})$. \\
  \ \ \ \ \ \ \ \ \ \ So by (1'), (2'),.., (n-1') and $<_1$-connectedness
  $h_{\rho} (\alpha_n) \in \tmop{Class} (n)$.

  The previous shows that the set (remember that $\alpha_n = \alpha$)\\
  $A \assign \{h_{\rho} (\alpha) | \rho \in \alpha \wedge h_{\rho} : B_{\rho}
  \longrightarrow \alpha \text{ is an } ( <, <_1, +) \text{-iso sucht hat } 
  h_{\rho} |_{\alpha} = \tmop{Id}_{\alpha}
  \} \subset \tmop{Class} (n)$ contains for any $\rho \in \alpha$ an element
  $h_{\rho} (\alpha) = h_{\rho} (\alpha_n) \in \tmop{Class} (n)$; moreover,
  since $\rho < \alpha = \alpha_n$ implies $\rho = h_{\rho} (\rho) < h_{\rho}
  (\alpha_n) = h_{\rho} (\alpha) < \alpha$, then $A$ is confinal in $\alpha$.
  Hence $\alpha \in \tmop{Lim} \tmop{Class} (n)$.

  {\noindent}4.\\
  We proceed by induction on $[2, \omega)$.
  
  Take $n \in [2, \omega)$.
  
  Suppose $\forall l \in n \cap [2, \omega) . \tmop{Class} (l) \subset
  \tmop{Lim} \tmop{Class} (l - 1)$. \ \ \ \ \ \ \ {\tmstrong{(cIH)}}

  If $\tmop{Class} (n) = \emptyset$ \ then we are done. So suppose
  $\tmop{Class} (n) \neq \emptyset$ and take $\alpha \in \tmop{Class} (n)$. By
  definition this means $\alpha \in \tmop{Class} (n - 1) \ni \alpha (+^{n -
  1})$ and $\alpha <_1 \alpha (+^{n - 1})$. \ \ \ \ \ \ \ {\tmstrong{(*3)}}

  Case $n - 1 = 1$.
  
  Then by (*3), the inequality $\alpha < \alpha 2 + 1 < \alpha (+^1$) and
  $<_1$-connectedness we get $\alpha <_1 \alpha 2 + 1$. Then by
  {\cite{GarciaCornejo1}} corollary \ref{eta(t)+1<less>_1_Equivalences},
  $\alpha \in \tmop{Lim} \mathbbm{E} = \tmop{Lim} \tmop{Class} (1)$.

  Case $n - 1 \geqslant 2$.
  
  By (*3) we know $\alpha \in \tmop{Class} (n - 1)$. Then, by 2., the on $[0,
  n - 2]$ recursively defined sequence of ordinals $\beta_{n - 1} \assign
  \alpha$, $\beta_{n - 1 - (k + 1)} \assign \beta_{n - 1 - k} (+^{n - 1 - (k +
  1)}$) satisfies \\
  $\beta_{n - 1} <_1 \beta_{n - 2} <_1 \ldots <_1 \beta_2 <_1 \beta_1 <_1
  \beta_1 2$ and $\forall i \in [1, n - 1] . \beta_i \in \tmop{Class} (i)$. \
  \ \ \ \ \ \ {\tmstrong{(*4)}}

  Let $\gamma \assign \alpha (+^{n - 1}) \in \tmop{Class} (n - 1)$.
  
  We show now by a side induction that $\forall u \in [2, n - 1] . \beta_{n -
  u} < \gamma$. \ \ \ \ \ \ \ {\tmstrong{(**3**)}}

  Let $u \in [2, n - 1]$.
  
  Suppose for $l \in u \cap [2, n - 1] . \beta_{n - l} < \gamma$. \ \ \ \ \ \
  \ {\tmstrong{(SIH)}}
  
  Since $\gamma \in \tmop{Class} (n - 1)$, then by (3*) (in case $u = 2$), by
  our SIH (in case $u > 2$) and 1. we have that $\beta_{n - (u - 1)} < \gamma
  \in \tmop{Class} (n - (u - 1))$, that is, $\gamma \in \{e \in \tmop{Class}
  (n - (u - 1)) | \beta_{n - (u - 1)} <_1 e\}$. But $\beta_{n - u} = \beta_{n
  - (u - 1)} (+^{n - u}) = \min \{e \in \tmop{Class} (n - u) | \beta_{n - (u -
  1)} < e\}$. From all this follows \\
  $\beta_{n - u} \leqslant \gamma$. \ \ (*5)
  
  On the other hand, our cIH applied to $\gamma \in \tmop{Class} (n - (u -
  1))$ implies $\gamma \in \tmop{Lim} \tmop{Class} (n - u)$; however, since
  $\beta_{n - u} = \min \{e \in \tmop{Class} (n - u) | \beta_{n - (u - 1)} <
  e\}$, then $\beta_{n - u} \nin \tmop{Lim} \tmop{Class} (n - u)$. From this
  and (*5) follows $\beta_{n - u} < \gamma$. This shows (**3**).

  From the fact that $\gamma \in \mathbbm{E}$, (**3**) and (*4) we have\\
  $\alpha = \beta_{n - 1} <_1 \beta_{n - 2} <_1 \ldots <_1 \beta_2 <_1 \beta_1
  <_1 \beta_1 2 < \beta_1 2 + 1 < \gamma = \alpha (+^{n - 1}$); moreover, from
  this, (*3) and $<_1$-connectedness we obtain $\beta_{n - 1} <_1 \beta_1 2 +
  1$. This way, by 3., it follows that \\
  $\alpha = \beta_{n - 1} \in \tmop{Lim} \tmop{Class} (n - 1)$ as we needed to
  show.
\end{proof}

\begin{proposition}
  \label{d<less>j<less>n_implies_Class(j)_contained_Lim(Class(d))}Let $j \in
  [2, \omega)$ and $c \in \tmop{Class} (j)$. Then for any $d \in [1, j)$ there
  exists a sequence $(c_{\xi})_{\xi \in X} \subset \tmop{Class} (d)$ such that
  $c_{\xi}\underset{\text{cof}}{\longhookrightarrow}c$.
\end{proposition}

\begin{proof}
  {\color{orange} Left to the reader.}
\end{proof}

\begin{corollary}
  \label{alpha_i_properties}For any $n \in [1, \omega)$ and any $\alpha \in
  \tmop{Class} (n)$ define by recursion on $[0, n - 1]$ the ordinals $\alpha_n
  \assign \alpha$, $\alpha_{n - (k + 1)} \assign \alpha_{n - k} (+^{n - (k +
  1)})$. Then
  \begin{enumeratealpha}
    \item $\alpha = \alpha_n \in \tmop{Class} (n)$.
    
    \item $\forall j \in [1, n - 1] . \alpha_j \in \tmop{Class} (j) \backslash
    \tmop{Class} (j + 1)$.
    
    \item $\alpha = \alpha_n <_1 \alpha_{n - 1} <_1 \alpha_{n - 2} <_1
    \alpha_{n - 3} <_1 \ldots <_1 \alpha_2 <_1 \alpha_1 <_1 \alpha_1 2 <
    \alpha (+^n)$.
    
    \item $\forall j \in [1, n - 1] .m (\alpha_j) = \alpha_1 2$.
  \end{enumeratealpha}
\end{corollary}

\begin{proof}
  {\color{orange} Left to the reader.}
\end{proof}

\begin{proposition}
  \label{<less>_1_chain_of_length_j}{\tmdummy}
  
  \begin{enumeratenumeric}
    \item $\forall j \in [1, \omega) \forall (a_1, \ldots, a_j) \in
    \tmop{OR}^j$.$a_j <_1 a_{j - 1} <_1 \ldots <_1 a_1 <_1 a_1 2
    \Longrightarrow a_j \in \tmop{Class} (j)$
    
    \item $\forall j \in [1, \omega) \forall a \in \tmop{OR}$.\\
    $a \in \tmop{Class} (j) \Longleftrightarrow \exists (a_1, \ldots, a_j) \in
    \tmop{OR}^j .a = a_j <_1 a_{j - 1} <_1 \ldots <_1 a_2 < a_1 < a_1 2$
  \end{enumeratenumeric}
\end{proposition}

\begin{proof}

  {\noindent}1.\\
  By induction on $[1, \omega)$.
  
  For $j = 1$ it is clear.

  Suppose for $j \in [1, \omega)$ the claim holds, that is \\
  \ \ \ $\forall (a_1, \ldots, a_j) \in \tmop{OR}^j$.$a_j <_1 a_{j - 1} <_1
  \ldots <_1 a_1 <_1 a_1 2 \Longrightarrow a_j \in \tmop{Class} (j)$. \ \ \
  {\tmstrong{(IH)}}

  We show that the claim holds for $j + 1$.

  Let $(a_1, \ldots, a_{j + 1}) \in \tmop{OR}^{j + 1}$ be such that $a_{j + 1}
  <_1 a_j <_1 \ldots <_1 a_1 <_1 a_1 2$. \ \ \ \ \ \ \ {\tmstrong{(*)}}
  
  By our (IH) follows $a_j \in \tmop{Class} (j)$ and therefore, $\forall s \in
  [1, j] .a_j \in \tmop{Class} (s)$. \ \ \ \ \ \ \ {\tmstrong{(**)}}
  
  Now, observe the following argument:
  
  $a_{j + 1} < 2 a_{j + 1}\underset{\text{because }a_{j + 1} < a_j \in \mathbbm{E}
  \subset \mathbbm{P}}{<}a_j$ and then $a_{j + 1} <_1 a_{j + 1} 2$ by (*)
  and $<_1$-connectedness; that is, $a_{j + 1} \in \mathbbm{E} = \tmop{Class}
  (1)$. But then, $a_{j + 1} < a_{j + 1} (+^1)\underset{\text{because }a_{j + 1} <
  a_j \in \mathbbm{E}}{\leqslant}a_j$ and then $a_{j + 1} <_1 a_{j + 1}
  (+^1$) by (*) and $<_1$-connectedness; that is, $a_{j + 1} \in \tmop{Class}
  (2)$. But then \\
  $a_{j + 1} < a_{j + 1} (+^2)\underset{\text{because }a_{j + 1} < a_j \in
  \tmop{Class} (2)}{\leqslant}a_j$ and then $a_{j + 1} <_1 a_{j + 1}
  (+^2$) by (*) and $<_1$-connectedness; that is, $a_{j + 1} \in \tmop{Class}
  (3)$. Inductively, we get $a_{j + 1} \in \tmop{Class} (j)$, \\
  $a_{j + 1} < a_{j + 1} (+^j)\underset{\text{because }a_{j + 1} < a_j \in
  \tmop{Class} (j)}{\leqslant}a_j$ and then, by (*) and
  $<_1$-connectedness, $a_{j + 1} <_1 a_{j + 1} (+^j$); that is, $a_{j + 1}
  \in \tmop{Class} (j + 1)$.

  {\noindent}2.\\
  {\color{orange} Left to the reader.}
\end{proof}

\begin{proposition}
  \label{alpha<less>_1_beta_in_Class(k)_then_alpha_in_Class(k+1)}$\forall k
  \in [1, \omega) . \forall \alpha, \beta \in \tmop{OR} . \alpha <_1 \beta \in
  \tmop{Class} (k) \Longrightarrow \alpha \in \tmop{Class} (k + 1)$.
\end{proposition}

\begin{proof}

  By induction on $[1, \omega)$.
  
  Let $k = 1$ and $\alpha, \beta \in \tmop{OR}$ be such that $\alpha <_1 \beta
  \in \tmop{Class} (1)$. Then $\alpha <_1 \alpha 2$ by
  $\leqslant_1$-connectedness (because $\alpha < \alpha 2 < \beta$), which, as
  we know, means $\alpha \in \mathbbm{E}$. This way $\alpha < \alpha (+^1)
  \leqslant \beta$, and then, by $\leqslant_1$-connectedness again, $\alpha
  <_1 \alpha (+^1$), that is, $\alpha \in \tmop{Class} (2)$.
  
  Suppose the claim holds for $k \in [1, \omega)$. \ \ \ \ \ \ \
  {\tmstrong{(IH)}}
  
  Let $\alpha, \beta \in \tmop{OR}$ be such that $\alpha <_1 \beta \in
  \tmop{Class} (k + 1)$. Then $\beta \in \tmop{Class} (k)$ by proposition
  \ref{Class(n)_first_properties}. So $\alpha <_1 \beta \in \tmop{Class} (k)$,
  and our (IH) implies $\alpha \in \tmop{Class} (k + 1)$. But then $\alpha <
  \alpha (+^{k + 1}) \leqslant \beta$, which implies by
  $\leqslant_1$-connectedness that $\alpha <_1 \alpha (+^{k + 1}$). Thus
  $\alpha \in \tmop{Class} (k + 2)$.
\end{proof}

\begin{proposition}
  \label{t_in_(a,a(+)..(+^1)2]_then_m(t)<less>=a(+)..(+^1)2}Let $i \in [1,
  \omega)$ and $\alpha \in \tmop{Class} (i)$. Then
  \begin{enumeratenumeric}
    \item $\forall z \in [1, i) .m (\alpha ( +^{i - 1}) ( +^{i - 2}) \ldots (
    +^z)) = \alpha (+^{i - 1}) ( +^{i - 2}) \ldots ( +^z) ( +^{z - 1}) \ldots
    ( +^1) 2$.
    
    \item $\forall t \in (\alpha, \alpha ( +^{i - 1}) ( +^{i - 2}) \ldots (
    +^2) ( +^1) 2] .m (t) \leqslant \alpha (+^{i - 1}) ( +^{i - 2}) \ldots (
    +^2) ( +^1) 2$.
  \end{enumeratenumeric}
\end{proposition}

\begin{proof}
  {\color{orange} Left to the reader.}
\end{proof}

\subsection{More general substitutions}

For our subsequent work we need to introduce the following general notion of
substitutions.

\begin{definition}
  \label{NoAppendix_simultaneous_substitution}Let $x \in \tmop{OR}$ and $f :
  \tmop{Dom} f \subset \mathbbm{E} \longrightarrow \mathbbm{E}$ a strictly
  increasing function such that $\tmop{Ep} (x) \subset \tmop{Dom} f$. We
  define $x [f]$, the ``simultaneous substitution of all the epsilon numbers
  $\tmop{Ep} (x)$ of the Cantor Normal Form of $x$ by the values of $f$ on
  them'', as
  
  $x [f] = \left\{ \begin{array}{l}
    f (x)  \text{ if } x \in \mathbbm{E}\\
    \\
    \sum_{i = 1}^n T_i [f] t_i \text{ if } x \nin \mathbbm{E}
    \text{ and } x =_{\tmop{CNF}} \sum_{i = 1}^n T_i t_i
    \text{ and } (t_1 \geqslant 2 \vee n \geqslant 2)\\
    \\
    \omega^{Z [f]}  \text{ if } x \nin \mathbbm{E} \text{ and } x =
    \omega^Z \text{ for some } Z \in \tmop{OR}\\
    \\
    x \text{ if } x < \varepsilon_0
  \end{array} \right.$

  Moreover, for $\tmop{Ep} (x) = \{e_1 > \ldots > e_k \}$ and a set $Y \assign
  \{\sigma_1 > \ldots > \sigma_k \} \subset \mathbbm{E}$ of epsilon numbers,
  we may also write $x [\tmop{Ep} (x) \assign Y]$ instead of $x [h]$, where $h
  : \tmop{Ep} (x) \longrightarrow Y$ is the function $h (e_i) \assign
  \sigma_i$.
\end{definition}

\begin{definition}
  Let $S \subset \tmop{OR}$ and $f_1, f_2 : S \longrightarrow \tmop{OR}$. We
  will denote as usual:
  \begin{itemizedot}
    \item $f_1 \leqslant f_2$ \ $: \Longleftrightarrow$ $\forall e \in S.f_1
    (e) \leqslant f_2 (e)$.
    
    \item $f_1 < f_2$ $: \Longleftrightarrow$ $f_1 \leqslant f_2 \wedge
    \exists e \in S.f_1 (e) < f_2 (e)$.
  \end{itemizedot}
\end{definition}

Now we enunciate the properties about these kind of substitutions that are of
our interest.

\begin{proposition}
  \label{NoAppendix_If_x<less>y_then_x[f]<less>y[f]}Let $x, y \in \tmop{OR}$.
  Let $f : S \subset \mathbbm{E} \longrightarrow \mathbbm{E}$ be a strictly
  increasing function with \ $\tmop{Ep} (x) \cup \tmop{Ep} (y) \subset S$.
  Then
  \begin{enumeratenumeric}
    \item $y \in \mathbbm{P} \Longleftrightarrow y [f] \in \mathbbm{P}$.
    
    \item $x < y \Longleftrightarrow x [f] < y [f]$.
    
    \item $y \in \mathbbm{E} \Longleftrightarrow y [f] \in \mathbbm{E}$ and $y
    \in \mathbbm{P} \backslash \mathbbm{E} \Longleftrightarrow y [f] \in
    \mathbbm{P} \backslash \mathbbm{E}$
  \end{enumeratenumeric}
\end{proposition}

\begin{proof}
  {\color{orange} Left to the reader.}
\end{proof}

\begin{proposition}
  \label{NoAppendix_[f]_=-compatible}Let $f : \tmop{Dom} f \subset \mathbbm{E}
  \longrightarrow \mathbbm{E}$ be a strictly increasing function. \\
  Let $A \assign \{x \in \tmop{OR} | \tmop{Ep} (x) \subset \tmop{Dom} f\}$.
  Then the assignation $\varphi : A \longrightarrow \tmop{OR}$ defined as
  $\varphi (x) \assign x [f]$ is a function with respect to the equality in
  the ordinals, that is, $\forall x, y \in A.x = y \Longrightarrow \varphi (x)
  = \varphi (y)$.
\end{proposition}

\begin{proof}
  {\color{orange} Left to the reader.}
\end{proof}

\begin{proposition}
  \label{NoAppendix_general_substitution_properties2}Let $x \in \tmop{OR}$ and
  $f : S \subset \mathbbm{E} \longrightarrow \mathbbm{E}$ be a strictly
  increasing function, where\\
  $\tmop{Ep} (x) \subset S$. Then
  \begin{enumeratenumeric}
    \item $x [f]$ is already in Cantor Normal Form.
    
    \item $\tmop{Ep} (x [f]) = f [\tmop{Ep} (x)] \subset \tmop{Im} f$.
    
    \item It exists $f^{- 1} : \tmop{Im} f \longrightarrow S$, $f^{- 1}$ is
    strictly increasing and $(x [f]) [f^{- 1}] = x$.
    
    \item Let $\alpha \in \mathbbm{E}$. Then $x \in [\alpha, \alpha ( +^1))
    \Longleftrightarrow \alpha \in S \wedge x [f] \in [f (\alpha), f (\alpha)
    ( +^1))$.
  \end{enumeratenumeric}
\end{proposition}

\begin{proof}
  {\color{orange} Left to the reader.}
\end{proof}

\begin{proposition}
  Let $f, g : S \subset \mathbbm{E} \longrightarrow \mathbbm{E}$ be strictly
  increasing functions. \\
  Let $D \assign \{e \in S|f (e) < g (e)\}$. Then
  \begin{enumeratenumeric}
    \item $f \leqslant g \Longleftrightarrow \forall x \in \tmop{OR} .
    \tmop{Ep} (x) \subset S \Longrightarrow x [f] \leqslant x [g])$.
    
    \item $f < g \Longrightarrow \forall x \in \tmop{OR} . (\tmop{Ep} (x)
    \subset S \wedge \tmop{Ep} (x) \cap D \neq \emptyset) \Longrightarrow x
    [f] < x [g])$.
  \end{enumeratenumeric}
\end{proposition}

\begin{proof}
  {\color{orange} Left to the reader.}
\end{proof}

\begin{proposition}
  \label{NoAppendix_operations_iso_gen_subst}Let $x, y \in \tmop{OR}$. Let $f
  : S \subset \mathbbm{E} \longrightarrow \mathbbm{E}$ be a strictly
  increasing function with $\tmop{Ep} (x) \cup \tmop{Ep} (y) \subset S$. Then
  \begin{enumeratenumeric}
    \item $\tmop{Ep} (x + y) \cup \tmop{Ep} (\omega^x) \cup \tmop{Ep} (x \cdot
    y) \subset S$
    
    \item $(x + y) [f] = x [f] + y [f]$
    
    \item $\omega^x [f] = \omega^{x [f]}$
    
    \item $(x \cdot y) [f] = x [f] \cdot y [f]$
  \end{enumeratenumeric}
\end{proposition}

\begin{proof}
  {\color{orange} Left to the reader.}
\end{proof}

\begin{proposition}
  Let $g : S \subset \mathbbm{E} \longrightarrow Z \subset \mathbbm{E}$ and $f
  : Z \subset \mathbbm{E} \longrightarrow \mathbbm{E}$ be strictly increasing
  functions. Then $f \circ g : S \subset \mathbbm{E} \longrightarrow
  \mathbbm{E}$ is strictly increasing and for any $t \in \tmop{OR}$ with
  $\tmop{Ep} (t) \subset S$, $t [f \circ g] = t [g] [f]$.
\end{proposition}

\begin{proof}
  {\color{orange} Left to the reader.}
\end{proof}

\subsection{The main theorem.}

Now we introduce certain notions that are necessary to enunciate the main
theorem.

\begin{definition}
  For $k \in [1, \omega)$, $\alpha \in \tmop{Class} (k)$ and $t \in [\alpha,
  \alpha ( +^k))$, the ordinal \\
  $\eta (k, \alpha, t)$ is defined as \\
  $\eta (k, \alpha, t) \assign \left\{ \begin{array}{l}
    \alpha (+^{k - 1}) ( +^{k - 2}) \ldots ( +^2) ( +^1) 2 \text{ iff } t
    \in [\alpha, \alpha ( +^{k - 1}) ( +^{k - 2}) \ldots ( +^2) ( +^1) 2]\\
    \\
    \max \{m (e) |e \in (\alpha, t]\} \text{ iff } t > \alpha (+^{k - 1})
    ( +^{k - 2}) \ldots ( +^2) ( +^1) 2
  \end{array} \right.$
\end{definition}

Our next proposition \ref{eta(k,alpha,t)_is_well_defined} shows that $\eta (k,
\alpha, t)$ is well defined.

\begin{proposition}
  \label{eta(k,alpha,t)_is_well_defined}Let $k \in [1, \omega)$, $\alpha \in
  \tmop{Class} (k)$, $t \in (\alpha, \alpha ( +^k))$ and \\
  $P \assign \{r \in (\alpha, t] |m (r) \geqslant t\}$. Then
  
  (0). $P$ is finite; more specifically $1 \leqslant |P| \leqslant k + 1$.
  
  (1). $\max \{m (e) |e \in P\}$ exists and $\max \{m (e) |e \in P\} = \max
  \{m (e) |e \in (\alpha, t]\}$.
  
  (2). $\eta (k, \alpha, t)$ is well defined.
  
  (3). $\eta (k, \alpha, t) \geqslant m (t) \geqslant t$.
\end{proposition}

\begin{proof}
  Consider $i, \alpha$ and $t$ as stated.
  
  {\noindent}(0)\\
  Clearly $t \in P$. So $|P| \geqslant 1$. We prove the other inequality by
  contradiction. Suppose $|P| \geqslant k + 2$. Then there exist $k + 1$
  ordinals $E_0, E_1, \ldots, E_{k - 1}, E_k \in P$ such that $E_k < E_{k - 1}
  < \ldots < E_1 < E_0 < t$; that is, $\forall l \in [0, k] .E_k < \ldots <
  E_l < \ldots < E_0 < t \leqslant m (E_l)$, and therefore, by
  $\leqslant_1$-connectedness, we get:

  a. $E_0 <_1 E_0 + 1 \leqslant t$, that is, $E_0 \in \tmop{Lim} \mathbbm{P}
  \subset \mathbbm{P}$.
  
  b. $E_1 < 2 E_1 < E_0 < t \leqslant m (E_1)$; then, by
  $\leqslant_1$-connectedness $E_1 <_1 2 E_1$, that is, $E_1 \in \mathbbm{E}$.
  
  c. $\alpha < E_k <_1 E_{k - 1} <_1 \ldots <_1 E_1 <_1 E_1 2 < E_0 < t <
  \alpha (+^k)$

  This way, from c. and proposition \ref{<less>_1_chain_of_length_j} follows
  $E_k \in \tmop{Class} (k) \cap (\alpha, \alpha ( +^k))$. Contradiction.

  Therefore $|P| \leqslant k + 1$.

  {\noindent}(1)\\
  Since by (0) $P \neq \emptyset$ is finite, then $\{m (e) |e \in P\}$ is
  finite too and thus $\mu \assign \max \{m (e) |e \in P\}$ exists. Then:

  $(I)$. \ $\mu \geqslant m (t) \geqslant t$ because $t \in P$ (and because $m
  (\beta) \geqslant \beta$ for any ordinal).
  
  $(I I)$. Since $P \subset (\alpha, t]$, then $\mu \in \{m (e) |e \in
  (\alpha, t]\}$.

  On the other hand, let $e \in (\alpha, t]$ be arbitrary. If $m (e) < t$,
  then $m (e) < \mu$ because of $(I)$. If $m (e) \geqslant t$, then $e \in P$
  and then $m (e) \leqslant \mu$. This shows that $\forall e \in (\alpha, t]
  .m (e) \leqslant \mu$ and since by $(I I)$ \\
  $\mu \in \{m (e) |e \in (\alpha, t]\}$, we have shown $\mu = \max \{m (e) |e
  \in (\alpha, t]\}$.

  {\noindent}(2)\\
  If $t \in [\alpha, \alpha ( +^{k - 1}) ( +^{k - 2}) \ldots ( +^2) ( +^1)
  2]$, then it is clear that $\eta (k, \alpha, t)$ is well defined. So suppose
  $t \in (\alpha ( +^{k - 1}) ( +^{k - 2}) \ldots ( +^2) ( +^1) 2, \alpha (
  +^k))$. By (1) $\max \{m (e) |e \in P\}$ exists and \\
  $\max \{m (e) |e \in P\} = \max \{m (e) |e \in (\alpha, t]\}
  \underset{\text{by definition}}{=} \eta (k, \alpha, t)$. That
  is, $\eta (k, \alpha, t)$ exists.

  {\noindent}(3)\\
  For $t > \alpha (+^{k - 1}) \ldots ( +^2) ( +^1) 2$ the assertion is clear.
  For $t \leqslant \alpha (+^{k - 1}) \ldots ( +^2) ( +^1) 2$, we get by
  proposition \ref{t_in_(a,a(+)..(+^1)2]_then_m(t)<less>=a(+)..(+^1)2}, $t
  \leqslant m (t) \leqslant \alpha (+^{k - 1}) ( +^{k - 2}) \ldots ( +^2) (
  +^1) 2 = \eta (k, \alpha, t)$.
\end{proof}

\begin{remark}
  \label{remark_eta(1,alpha,t)=eta(t)}The ordinal $\eta (i, \alpha, t)$ is
  meant to play in $\tmop{Class} (i)$ the analogous role that the ordinal
  $\eta t$ played in $\tmop{Class} (1) = \mathbbm{E}$. Particularly, for $i =
  1, \alpha \in \tmop{Class} (1) = \mathbbm{E}$ and $t \in [\alpha, \alpha (
  +^1))$, \\
  $\eta (1, \alpha, t)\underset{\text{by {\cite{GarciaCornejo1}} proposition }
  \ref{eta(t)_m(t)_and_<less>_1}}{=}\eta t$.
\end{remark}

\begin{proposition}
  $\forall i \in [1, \omega) \forall \alpha \in \tmop{Class} (i) \forall t \in
  [\alpha, \alpha ( +^i)) . \eta (i, \alpha, t) \in (\alpha, \alpha ( +^i))$.
\end{proposition}

\begin{proof}
  {\color{orange} Left to the reader.}
\end{proof}

\begin{definition}
  \label{l(t)}Let $i \in [1, \omega)$, $\alpha \in \tmop{Class} (i)$ and $t
  \in [\alpha, \alpha ( +^i))$. We define \\
  $l (i, \alpha, t) \assign \left\{ \begin{array}{l}
    \alpha (+^{i - 1}) ( +^{i - 2}) \ldots ( +^2) ( +^1) 2 \text{ iff } t
    \in [\alpha, \alpha ( +^{i - 1}) ( +^{i - 2}) \ldots ( +^2) ( +^1) 2]\\
    \\
    \min \{r \in (\alpha, t] | m (r) = \eta (i, \alpha, t)\} \text{ iff }
    t > \alpha (+^{i - 1}) ( +^{i - 2}) \ldots ( +^2) ( +^1) 2
  \end{array} \right.$.
\end{definition}

\begin{proposition}
  \label{eta(i,a,l(i,a,t))=eta(i,a,t)_if_t_in_(a(+)...(+^1)2,a(+^i)}Let $i \in
  [1, \omega)$, $\alpha \in \tmop{Class} (i)$ and $t \in (\alpha ( +^{i - 1})
  ( +^{i - 2}) \ldots ( +^2) ( +^1) 2, \alpha ( +^i))$. Then
  \begin{enumeratenumeric}
    \item $l (i, \alpha, t) > \alpha (+^{i - 1}) \ldots ( +^1) 2$
    
    \item $\eta (i, \alpha, l (i, \alpha, t)) = \max \{m (e) |e \in (\alpha, l
    (i, \alpha, t)]\} = m (l (i, \alpha, t)) = \eta (i, \alpha, t)$.
  \end{enumeratenumeric}
\end{proposition}

\begin{proof}
  {\color{orange} Left to the reader.}
\end{proof}

\begin{remark}
  \label{remark_l(1,alpha,t)}With respect to definition \ref{l(t)}, consider
  the case $i = 1$. Let $t \in (\alpha 2, \alpha^+)$ and suppose $l (1,
  \alpha, t) \in (\alpha, t)$. The inequalities $l (1, \alpha, t) \leqslant t
  < \eta t = m (l (1, \alpha, t))$ and the fact that \\
  $l (1, \alpha, t) \nin \mathbbm{E}$ imply, by {\cite{GarciaCornejo1}}
  theorem \ref{Wilken_Theorem1} and {\cite{GarciaCornejo1}} corollary
  \ref{Wilken_Corollary1}, that $\mathbbm{P} \ni l (1, \alpha, t) \wedge m (l
  (1, \alpha, t)) < l (1, \alpha, t) 2$. Therefore $\mathbbm{P} \ni l (1,
  \alpha, t) \leqslant t < m (l (1, \alpha, t)) < l (1, \alpha, t) 2$, which
  subsequently implies (by considering the cantor normal form of $t$) that
  $\pi t = l (1, \alpha, t)$. From this we conclude: \\
  For any $s \in [\alpha, \alpha^+)$, $l (1, \alpha, s) = \left\{
  \begin{array}{l}
    \alpha 2 \text{ iff } s \in [\alpha, \alpha 2]\\
    \pi s \text{ iff } s \in (\alpha 2, \alpha^+) \wedge l (1, \alpha, s)
    < s\\
    s \text{ iff } s \in (\alpha 2, \alpha^+) \wedge l (1, \alpha, s)
    \nless s
  \end{array} \right.$.
\end{remark}

\begin{definition}
  Let $i \in [1, \omega)$, $\alpha \in \tmop{Class} (i)$, $t \in [\alpha,
  \alpha ( +^i)$ and $j \in [1, i]$. We define $\lambda (j, t)$ as the only
  one ordinal $\delta$ satisfying $\delta \in \tmop{Class} (j)$ and $t \in
  [\delta, \delta ( +^j))$ in case such ordinal exists, and $- \infty$
  otherwise.
\end{definition}

\begin{remark}
  For $i \in [1, \omega)$, $\alpha \in \tmop{Class} (i)$, $t \in [\alpha,
  \alpha ( +^i))$ and $j \in [1, i]$, (i.e., all the conditions above)
  $\lambda (j, t)$ will always be an ordinal. Again, the reason to give the
  definition this way is just because the existence of $\lambda (j, t)$ is not
  completely obvious (we will see that later).
\end{remark}

We can now present the main theorem of this article.

\begin{theorem}
  \label{concise_gen_thm}For any $n \in [1, \omega)$,
  
  {\noindent}(1). $\tmop{Class} (n)$ is $\kappa$-club for any non-countable
  regular ordinal $\kappa$.

  There exist a binary relational $\leqslant^n \subset \tmop{Class} (n) \times
  \tmop{OR}$ such that:
  
  For $\alpha, c \in \tmop{Class} (n)$ and any $t \in \alpha (+^n)$ there
  exist
  
  - A finite set $T (n, \alpha, t) \subset \mathbbm{E} \cap \alpha (+^n)$,
  
  - A strictly increasing function $g (n, \alpha, c) : \tmop{Dom} g (n,
  \alpha, c) \subset \mathbbm{E} \cap \alpha (+^n) \longrightarrow \mathbbm{E}
  \cap c (+^n)$
  
  such that:

  {\noindent}(2) The function $H : (\tmop{Dom} g (n, \alpha, c)) \cap (\alpha,
  \alpha ( +^n)) \longrightarrow (c, c ( +^n))$, $H (e) \assign e [g (n,
  \alpha, c)]$ is \\
  an $(<, +, \cdot, <_1, \lambda x. \omega^x, (+^1), (+^2), \ldots, (+^{n -
  1}))$ isomorphism.

  {\noindent}(3) The relation $\leqslant^n$ satisfies
  $\leqslant^n$-connectedness, $\leqslant^n$-continuity and is such that\\
  $(t \in [\alpha, \alpha (+^n)] \wedge \alpha \leqslant^n t) \Longrightarrow
  \alpha \leqslant_1 t$.
  
  \ \
  
  {\noindent}(4) (First fundamental cofinality property of $\leqslant^n$). \\
  If $t \in [\alpha, \alpha ( +^n)) \wedge \alpha \leqslant^n t + 1$, then
  there exists a sequence $(c_{\xi})_{\xi \in X} \subset \alpha \cap
  \tmop{Class} (n)$ such that $c_{\xi}
  \underset{\text{cof}}{\longhookrightarrow} \alpha$, $\forall \xi \in X.T (n,
  \alpha, t) \cap \alpha \subset c_{\xi}$ and $c_{\xi} \leqslant_1 t [g (n,
  \alpha, c_{\xi})]$.

  {\noindent}(5). (Second fundamental cofinality property of $\leqslant^n$).\\
  Suppose $t \in [\alpha, \alpha ( +^n)) \wedge \alpha \in \tmop{Lim} \{\gamma
  \in \tmop{Class} (n) |T (n, \alpha, t) \cap \alpha \subset \gamma \wedge
  \gamma \leqslant_1 t [g (n, \alpha, \gamma)]\}$. Then \ \ \ \ \ \
  
  (5.1) $\forall s \in [\alpha, t + 1] . \alpha \leqslant^n s$, and therefore
  
  (5.2) $\alpha \leqslant_1 t + 1$

  {\noindent}(6). $(t \in [\alpha, \alpha (+^n)) \wedge \alpha <_1 \eta (n,
  \alpha, t) + 1) \Longrightarrow \alpha \leqslant^n \eta (n, \alpha, t) + 1$.
\end{theorem}

Theorem \ref{concise_gen_thm} states the general result we are striving for.
But the proof of theorem \ref{concise_gen_thm} is a very long journey: we need
to overcome many technical difficulties not stated in it; because of that, we
restate it in a more technical way: theorem \ref{most_most_general_theorem}.

\begin{theorem}
  \label{most_most_general_theorem}For any $n \in [1, \omega)$
  
  {\noindent}(0). $\tmop{Class} (n)$ is $\kappa$-club for any non-countable
  regular ordinal $\kappa$.

  {\noindent}(1). For any $\alpha \in \tmop{Class} (n)$ the functions
  \begin{itemizedot}
    \item $S (n, \alpha) : \tmop{Class} (n - 1) \cap (\alpha, \alpha ( +^n))
    \longrightarrow \tmop{Subsets} (\tmop{Class} (n - 1) \cap (\gamma, \gamma
    (+^n)))$
    
    {\noindent}$S (n, \alpha) (\delta) \assign \{e \in \tmop{Class} (n - 1)
    \cap (\alpha, \alpha (+^n)) \cap \delta | m (e) [g (n - 1, e, \delta)]
    \geqslant m (\delta)\}$
    
    \item $f (n, \alpha) : \tmop{Class} (n - 1) \cap (\alpha, \alpha ( +^n))
    \longrightarrow \tmop{Subsets} (\tmop{OR})$
    
    {\noindent}$f (n, \alpha) (\delta) \assign \left\{ \begin{array}{l}
      \{\delta\} \text{ iff } S (n, \alpha) (\delta) = \emptyset\\
      \\
      f (n, \alpha) (s) \cup \{\delta\}  \text{ iff } S (n, \alpha)
      (\delta) \neq \emptyset \wedge s \assign \sup (S (n, \alpha) (\delta))
    \end{array} \right.$
  \end{itemizedot}
  \ \ \ \ are well defined and are such that
  
  (1.1) If $S (n, \alpha) (\delta) \neq \emptyset$ then $\sup (S (n, \alpha)
  (\delta)) \in S (n, \alpha) (\delta) \subset \tmop{Class} (n - 1) \cap
  \delta$.
  
  (1.2) $\forall \delta \in \tmop{Class} (n - 1) \cap (\alpha, \alpha (
  +^n))$.$\delta \in f (n, \alpha) (\delta) \subset (\alpha, \alpha ( +^n))
  \cap \tmop{Class} (n - 1)$\\
  \hspace*{14mm} and $f (n, \alpha) (\delta)$ is finite.
  
  (1.3) $\forall q \in [1, \omega) . \forall \sigma \in (\alpha, \alpha (
  +^n)) \cap \tmop{Class} (n - 1)$. If $f (n, \alpha) (\sigma) = \{\sigma_1 >
  \ldots > \sigma_q \}$ for some \hspace*{14mm}
  $\sigma_1, \ldots, \sigma_q \in \tmop{OR}$ then
  
  \ \ \ \ \ \ (1.3.1) $\sigma_1 = \sigma$,
  
  \ \ \ \ \ \ (1.3.2) $q \geqslant 2 \Longrightarrow \forall j \in \{1,
  \ldots, q - 1\} .m (\sigma_j) \leqslant m (\sigma_{j + 1}) [g (n - 1,
  \sigma_{j + 1}, \sigma_j)]$ and
  
  \ \ \ \ \ \ (1.3.3) $\sigma_q = \min \{e \in (\alpha, \sigma_q] \cap
  \tmop{Class} (n - 1) | m (e) [g (n - 1, e, \sigma_q)] \geqslant m
  (\sigma_q)\}$.
  
  \ \ \ \ \ \ (1.3.4) $m (\sigma) = m (\sigma_1) \leqslant m (\sigma_2) [g (n
  - 1, \sigma_2, \sigma)] \leqslant \ldots \leqslant m (\sigma_q) [g (n - 1,
  \sigma_q, \sigma)]$.
  
  \ \ \ \ \ \ (1.3.5) $\sigma_q = \min \{e \in (\alpha, \alpha (+^n)) \cap
  \tmop{Class} (n - 1) |$ $e \leqslant \sigma_q \wedge m (e) [g (n - 1, e,
  \sigma_q)] \geqslant m (\sigma_q) \}$
  
  \ \ \ \ \ \ \ \ \ \ \ \ \ \ \ \ \ \ $= \min \{e \in (\alpha, \alpha (+^n))
  \cap \tmop{Class} (n - 1) |$ $e \leqslant \sigma_q \wedge m (e) [g (n - 1,
  e, \sigma_q)] = m (\sigma_q) \}$.
  
  \ \ \ \ \ \ (1.3.6) For any $j \in \{1, \ldots, q - 1\}$,
  
  \ \ \ \ \ \ \ \ \ \ \ \ \ $\sigma_j = \min \{e \in (\alpha, \alpha
  (+^n)) \cap \tmop{Class} (n - 1) |\sigma_{j +
  1} < e \leqslant \sigma_j \wedge m (e) [g (n - 1, e, \sigma_j)] \geqslant m
  (\sigma_j) \}$
  
  \ \ \ \ \ \ \ \ \ \ \ \ \ \ \ \ $= \min \{e \in (\alpha, \alpha (+^n))
  \cap \tmop{Class} (n - 1) | \sigma_{j +
  1} < e \leqslant \sigma_j \wedge m (e) [g (n - 1, e, \sigma_j)] = m
  (\sigma_j) \}$
  
  {\noindent}{\color{blue} Note: $S (1, \alpha) = \emptyset = f (1, \alpha)$.
  These functions are interesting for $n \geqslant 2$.}

  {\noindent}(2).
  
  (2.1) For any $\alpha \in \tmop{Class} (n)$ and any $t \in \alpha (+^n$)
  consider the set $T (n, \alpha, t)$ defined as:
  
  \ \ \ \ \ \ \ $T (n, \alpha, t) \assign \bigcup_{E \in \tmop{Ep} (t)} T (n,
  \alpha, E)$ \ \ \ if \ \ $t \nin \mathbbm{E}$;
  
  \ \ \ \ \ \ \ $T (n, \alpha, t) \assign \{t\}$ \ \ \ \ \ \ \ \ \ \ \ \ \ \
  \ \ \ \ \ \ if \ \ $t \in \mathbbm{E} \cap (\alpha + 1)$;

  \ \ \ \ \ \ \ $T (n, \alpha, t) \assign \bigcup_{i \in \omega} O (i, t)$, \
  \ \ \ \ \ \ \ \ \ \ if \ \ $t \in (\alpha, \alpha ( +^n)) \cap \mathbbm{E}$,
  \ \\
  \hspace*{13mm} where for $E \in (\alpha, \alpha ( +^n)) \cap
  \mathbbm{E}$:
  
  \ \ \ \ \ \ \ we define $E_1 \assign \lambda (1, m (E)), E_2 \assign
  \lambda (2, E_1), \ldots, E_n \assign \lambda (n, E_{n - 1})$,
  
  \ \ \ \ \ \ (note $\alpha = E_n \leqslant \ldots E_3 \leqslant E_2 < E_1$)
  and

  \ \ \ \ \ \ \ $O (0, E) \assign\underset{\delta \in W (0, k, E), k = 1, \ldots, n - 1}{\bigcup}
  f(k + 1, \lambda (k + 1, \delta)) (\delta)\cup\tmop{Ep} (m (\delta)) \cup \{\lambda (k + 1, \delta)\}$;

  \ \ \ \ \ \ \ $W (0, k, E) \assign (\alpha, \alpha ( +^n)) \cap \{E_1 > E_2
  \geqslant E_3 \geqslant \ldots \geqslant E_n = \alpha\} \cap (\tmop{Class}
  (k) \backslash \tmop{Class} (k + 1))$;

  \ \ \ \ \ \ \ $O (l + 1, E) \assign\underset{\delta \in W (l, k, E),k = 1, \ldots, n - 1}{\bigcup}
  f (k + 1, \lambda (k + 1, \delta)) (\delta) \cup
  \tmop{Ep} (m (\delta)) \cup \{\lambda (k + 1, \delta)\}$;

  \ \ \ \ \ \ \ $W (l, k, E) \assign (\alpha, \alpha ( +^n)) \cap O (l, E)
  \cap (\tmop{Class} (k) \backslash \tmop{Class} (k + 1))$.

  \ \ \ \ \ \ \ Then $T (n, \alpha, t) \subset \mathbbm{E} \cap \alpha (+^n)$
  is such that:

  \ \ \ \ \ \ \ (2.1.1) $\tmop{Ep} (t) \subset T (n, \alpha, t)$ and $T (n,
  \alpha, t)$ is finite.
  
  \ \ \ \ \ \ \ (2.1.2) $T (n, \alpha, t + 1) = T (n, \alpha, t)$
  
  \ \ \ \ \ \ \ (2.1.3) $\alpha (+^{n - 1}) ( +^{n - 2}) \ldots ( +^2) ( +^1)
  2 \leqslant t \Longrightarrow T (n, \alpha, \eta (n, \alpha, t)) \cap \alpha
  \subset T (n, \alpha, t) \cap \alpha$
  
  \ \ \ \ \ \ \ (2.1.4) $\alpha (+^{n - 1}) ( +^{n - 2}) \ldots ( +^2) ( +^1)
  2 \leqslant t \Longrightarrow T (n, \alpha, l (n, \alpha, t)) \subset T (n,
  \alpha, t)$ \\
  \ \ \ \ \ \ \ \ \ \ \ \ \ \ \ \ \ \ \ \
  
  (2.2) For any $\alpha, c \in \tmop{Class} (n)$ there exist a function \\
  \hspace*{14mm} $g (n, \alpha, c) : \tmop{Dom} g (n, \alpha, c)
  \subset \mathbbm{E} \cap \alpha (+^n) \longrightarrow \mathbbm{E} \cap c
  (+^n$) such that \\
  \hspace*{14mm} (2.2.1) $g (n, \alpha, c) |_{c \cap \alpha \cap (\tmop{Dom} g
  (n, \alpha, c))}$ and $g (n, \alpha, \alpha)$ are the identity functions in
  their \\ \hspace*{26mm} respective domain.
  
  \ \ \ \ \ \ \ (2.2.2) $g (n, \alpha, c)$ is strictly increasing.
  
  \ \ \ \ \ \ \ (2.2.3) $\forall t \in \alpha (+^n) .T (n, \alpha, t) \cap
  \alpha \subset c \Longleftrightarrow \tmop{Ep} (t) \subset \tmop{Dom} g (n,
  \alpha, c)$
  
  \ \ \ \ \ \ \ (2.2.4) $\forall t \in \alpha (+^n) . \tmop{Ep} (t) \subset
  \tmop{Dom} g (n, \alpha, c) \Longrightarrow T (n, c, t [g (n, \alpha, c)])
  \cap c = T (n, \alpha, t) \cap \alpha$
  
  \ \ \ \ \ \ \ \ \ (2.2.5) For any $t \in [\alpha, \alpha ( +^n))$ with
  $\tmop{Ep} (t) \subset \tmop{Dom} g (n, \alpha, c)$, $\tmop{Ep} (\eta (n,
  \alpha, t)) \subset \tmop{Dom} g (n, \alpha, c)$\\
  \hspace*{25mm} and $\eta (n, \alpha, t) [g (n,
  \alpha, c)] = \eta (n, c, t [g (n, \alpha, c)])$.

  By (2.2.2), $g (n, \alpha, c)$ is bijective in its image. Let's denote $g^{-
  1} (n, \alpha, c)$ to the inverse function \\ 
  \hspace*{4mm} of $g (n, \alpha, c)$.

  (2.3) For (2.3.1), (2.3.2) and (2.3.3) we suppose $c \leqslant \alpha$.
  Then
  
  \ \ \ \ \ \ \ (2.3.1) $\tmop{Dom} g (n, c, \alpha) = \mathbbm{E} \cap c
  (+^n$)
  
  \ \ \ \ \ \ \ (2.3.2) $g (n, \alpha, c) = g^{- 1} (n, c, \alpha)$
  
  \ \ \ \ \ \ \ (2.3.3) $g (n, \alpha, c) [\tmop{Dom} g (n, \alpha, c)] =
  \mathbbm{E} \cap c (+^n$)

  (2.4) $g (n, \alpha, c)$ has the following homomorphism-like properties:
  
  \ \ \ \ \ \ \ (2.4.1) $g (n, \alpha, c) (\alpha) = c$
  
  \ \ \ \ \ \ \ (2.4.2) For any $i \in [1, n]$ and any $e \in (\tmop{Dom} g
  (n, \alpha, c)) \cap [\alpha, \alpha ( +^n))$,\\
  \hspace*{26mm} $e \in \tmop{Class} (i)
  \Longleftrightarrow g (n, \alpha, c) (e) \in \tmop{Class} (i)$
  
  \ \ \ \ \ \ \ (2.4.3) The function $e \longmapsto e [g (n, \alpha, c)]$
  with domain $(\tmop{Dom} g (n, \alpha, c)) \cap (\alpha, \alpha ( +^n))$
  is\\
  \hspace*{26mm} an $(<, +, \cdot, <_1, \lambda x.
  \omega^x, (+^1), (+^2), \ldots, (+^{n - 1}))$ isomorphism
  
  \ \ \ \ \ \ \ (2.4.4) $\forall e \in (\tmop{Dom} g (n, \alpha, c)) \cap
  (\alpha, \alpha ( +^n)) .m (g (n, \alpha, c) (e)) = m (e) [g (n, \alpha,
  c)]$.
  
  \ \ \ \ \ \ \ (2.4.5) Suppose $n \geqslant 2$. Then\\
  \hspace*{25mm} $\forall i \in [2, n]$.$\forall e \in \tmop{Class} (i) \cap
  (\tmop{Dom} g (n, \alpha, c)) \cap [\alpha, \alpha ( +^n))$.\\
  \hspace*{25mm} $\forall E \in (e, e ( +^i)) \cap
  \tmop{Class} (i - 1)$.\\
  \hspace*{26mm} $f (i, e) (E) = \{E_1 > \ldots >
  E_q \} \Longleftrightarrow$\\
  \hspace*{26mm} $f (i, g (n, \alpha, c) (e)) (g
  (n, \alpha, c) (E)) = \{g (n, \alpha, c) (E_1) > \ldots > g (n, \alpha, c)
  (E_q)\}$
  
  \ \ \ \ \ \ \ (2.4.6) Suppose $n \geqslant 2$. Then \\
   \hspace*{26mm}$\forall i \in [2, n] . \forall s
  \in \tmop{Class} (i - 1) \cap [\alpha, \alpha ( +^n))$.$g (n, \alpha, c) (\lambda (i, s)) =
  \lambda (i, g (n, \alpha, c) (s))$

  (2.5) For (2.5.1), (2.5.2) and (2.5.3) we suppose $c \leqslant \alpha$.Then
  for all $d \in \tmop{Class} (n) \cap [c, \alpha]$,
  
  \ \ \ \ \ \ \ (2.5.1) $\tmop{Dom} g (n, \alpha, c) \subset \tmop{Dom} g (n,
  \alpha, d)$
  
  \ \ \ \ \ \ \ (2.5.2) $g (n, \alpha, d) [\tmop{Dom} g (n, \alpha, c)]
  \subset \tmop{Dom} g (n, d, c)$
  
  \ \ \ \ \ \ \ (2.5.3) $g (n, \alpha, c) = g (n, d, c) \circ g (n, \alpha,
  d) |_{\tmop{Dom} g (n, \alpha, c)}$ and therefore\\
  \ \ \ \ \ \ \ \ \ \ \ \ \ \ \ \ \ \ \ \ $g^{- 1} (n, \alpha, d) \circ g^{-
  1} (n, d, c) = g^{- 1} (n, \alpha, c) : \mathbbm{E} \cap c (+^n)
  \longrightarrow \tmop{Dom} g (n, \alpha, c)$.

  {\noindent}(3). There exists a binary relational $\leqslant^n \subset
  \tmop{Class} (n) \times \tmop{OR}$ satisfying $\leqslant^n$-connectedness
  and \\
  $\leqslant^n$-continuity such that $\forall \alpha \in \tmop{Class} (n) .
  \forall t \in [\alpha, \alpha ( +^n)] . \alpha \leqslant^n t \Longrightarrow
  \alpha \leqslant_1 t$; moreover:

  {\noindent}(4) (First fundamental cofinality property of $\leqslant^n$).\\
  Let $\alpha \in \tmop{Class} (n)$ and $t \in [\alpha, \alpha ( +^n))$ be
  arbitrary. If $\alpha \leqslant^n t + 1$, then there exists a sequence
  $(c_{\xi})_{\xi \in X} \subset \alpha \cap \tmop{Class} (n)$ such that
  $c_{\xi} \underset{\text{cof}}{\longhookrightarrow} \alpha$, $\forall \xi
  \in X.T (n, \alpha, t) \cap \alpha \subset c_{\xi}$ and $c_{\xi} \leqslant_1
  t [g (n, \alpha, c_{\xi})]$.

  {\noindent}(5). (Second fundamental cofinality property of $\leqslant^n$).\\
  Let $\alpha \in \tmop{Class} (n)$ and $t \in [\alpha, \alpha ( +^n))$.\\
  Suppose $\alpha \in \tmop{Lim} \{\gamma \in \tmop{Class} (n) |T (n, \alpha,
  t) \cap \alpha \subset \gamma \wedge \gamma \leqslant_1 t [g (n, \alpha,
  \gamma)]\}$. Then \ \ \ \ \ \
  
  (5.1) $\forall s \in [\alpha, t + 1] . \alpha \leqslant^n s$, and therefore
  
  (5.2) $\alpha \leqslant_1 t + 1$

  {\noindent}(6). For $\alpha \in \tmop{Class} (n)$ and $t \in [\alpha, \alpha
  ( +^n))$, $\alpha <_1 \eta (n, \alpha, t) + 1 \Longrightarrow \alpha
  \leqslant^n \eta (n, \alpha, t) + 1$
\end{theorem}

The proof of the previous theorem \ref{most_most_general_theorem} will be
carried out by induction on $([1, \omega), <)$, and one proves simultaneously
(0), (1), (2), (3), (4), (5) and (6). Indeed, such proof is now our current
goal.

\subsection{The case $n = 1$ of theorem \ref{most_most_general_theorem}}

\begin{proposition}
  \label{GenThm_holds_for_n=1}Theorem \ref{most_most_general_theorem} holds
  for $n = 1$.
\end{proposition}

\begin{proof}

  {\noindent}(0).\\
  $\mathbbm{E}$ is $\kappa$-club for any non-countable regular ordinal
  $\kappa$.

  {\noindent}(1).\\
  Let $\alpha \in \mathbbm{E} = \tmop{Class} (1)$. We define $S (1, \alpha)
  \assign \emptyset$ and $f (1, \alpha) \assign \emptyset$. Then clearly $S
  (1, \alpha)$ and $f (1, \alpha)$ satisfy the properties stated.

  {\noindent}(2).

  (2.1)\\
  Let $\alpha, c \in \mathbbm{E} = \tmop{Class} (1)$ with $c \leqslant
  \alpha$. Let $t \in \alpha (+^1$). Note $T (1, \alpha, c) = \tmop{Ep} (t)
  \subset \mathbbm{E}$ is well defined and clearly $(2.1.1)$ and $(2.1.2)$
  hold. Now, suppose $t \geqslant \alpha 2$. Since \\
  $\eta (1, \alpha, t) = \eta t = \max \{t, \pi t + d \pi t\}$, then $(2.1.3)$
  holds. Finally, by remark \ref{remark_l(1,alpha,t)}, \\
  $l (1, \alpha, t) \in \{\alpha 2, \pi t, t\}$ and therefore (2.1.4) holds
  too.

  (2.2)\\
  Let $\alpha, c \in \mathbbm{E} = \tmop{Class} (1)$. Consider $\tmop{Dom} g
  (1, \alpha, c) \assign ( \mathbbm{E} \cap c \cap \alpha) \cup \{\alpha\}$
  and \\ 
  $g (1, \alpha, c) : \tmop{Dom} g (1, \alpha, c) \longrightarrow \mathbbm{E}
  \cap c (+^1$) be the function defined as $e$ $\longmapsto$ $e$ for 
  $e \in \mathbbm{E} \cap c \cap \alpha$ and $\alpha \longmapsto$ $c$.
  
  Then it is easy to see that $g (1, \alpha, c)$ satisfies (2.2.1), (2.2.2)
  and (2.2.3). Besides, $g (1, \alpha, c)$ also satisfies (2.2.4): Take $t \in
  \alpha^+$ with $\tmop{Ep} (t) \subset \tmop{Dom} g (1, \alpha, c) = (
  \mathbbm{E} \cap c \cap \alpha) \cup \{\alpha\}$. Then \\
  $\tmop{Ep} (t) \cap \alpha \subset c$ and $t [g (1, \alpha, c)] = t [\alpha
  \assign c]$ and therefore \\
  $T (1, c, t [g (1, \alpha, c)]) \cap c = T (1, c, t [\alpha \assign c]) \cap
  c = \tmop{Ep} (t [\alpha \assign c]) \cap c\underset{\text{by
  {\cite{GarciaCornejo1}} proposition
  \ref{[alpha:=e]_proposition3}}}{=}\tmop{Ep} (t) \cap \alpha = T (1,
  \alpha, t)$. Finally, we show that $g (1, \alpha, c)$ satisfies (2.2.5):
  Take $t \in \alpha^+$ with $\tmop{Ep} (t) \subset \tmop{Dom} g (1, \alpha,
  c)$. Then $\tmop{Ep} (t) \cap \alpha \subset c$ and so $\tmop{Ep} (\eta (1,
  \alpha, t)) \cap \alpha\underset{\text{remark
  \ref{remark_eta(1,alpha,t)=eta(t)}}}{=}\tmop{Ep} (\eta t) \cap
  \alpha\underset{\text{by {\cite{GarciaCornejo1}} proposition
  \ref{pi.eta.substitutions}}}{\subset}c$, which means $\tmop{Ep} (\eta (1,
  \alpha, t)) \subset \tmop{Dom} g (1, \alpha, c)$. Moreover, $\eta (1,
  \alpha, t) [g (1, \alpha, c)] = (\eta t) [\alpha \assign c]\underset{\text{by
  {\cite{GarciaCornejo1}} proposition \ref{pi.eta.substitutions}}}{=}\eta (t
  [\alpha \assign c]) = \eta (1, c, t [\alpha \assign c]) = \eta (1, c, t [g
  (1, \alpha, c)])$.

  (2.3)\\
  Considering $\alpha, c \in \mathbbm{E}$ and $g (1, \alpha, c)$ as in (2.2)
  with the extra assumption $c \leqslant \alpha$ it is immediate that (2.3.1),
  (2.3.2) and (2.3.3) hold.

  (2.4)\\
  Given $\alpha, c \in \mathbbm{E}$ and $g (1, \alpha, c)$ as in (2.2), it is
  clear that (2.4.1), (2.4.2), (2.4.5) and (2.4.6) hold. Moreover, (2.4.3) and
  (2.4.4) are {\cite{GarciaCornejo1}} corollary
  \ref{A_[alpha:=e]_isomorphisms} and {\cite{GarciaCornejo1}} remark
  \ref{remark_m-iso_<less>_1-iso}.

  (2.5)\\
  Take $\alpha, d, c \in \mathbbm{E}$ with $c \leqslant d \leqslant \alpha$.
  Then $\tmop{Dom} g (1, \alpha, c) = ( \mathbbm{E} \cap c \cap \alpha) \cup
  \{\alpha\} \subset ( \mathbbm{E} \cap d \cap \alpha) \cup \{\alpha\} =
  \tmop{Dom} g (1, \alpha, d)$, that is, (2.5.1) holds. Moreover, \\
  $g (1, \alpha, d) [\tmop{Dom} g (1, \alpha, c)] = \{g (1, \alpha, d) (e) | e
  \in ( \mathbbm{E} \cap c \cap \alpha) \cup \{\alpha\}\} = ( \mathbbm{E} \cap
  c \cap \alpha) \cup \{d\} \subset$\\
  $( \mathbbm{E} \cap c \cap d) \cup \{d\} = \tmop{Dom} g (1, d, c)$, i.e.,
  (2.5.2) holds. Let's show that (2.5.3) also holds: For $e \in \tmop{Dom} g
  (1, \alpha, c) = ( \mathbbm{E} \cap c \cap \alpha) \cup \{\alpha\}$, $e
  \underset{g (1, \alpha, d)}{\longmapsto} \left\{ \begin{array}{l}
    e \underset{g (1, d, c)}{\longmapsto} c = g (1, \alpha, c) (e)
    \text{ iff } e = \alpha\\
    \\
    e \underset{g (1, d, c)}{\longmapsto} e = g (1, \alpha, c) (e)
    \text{ iff } e \neq \alpha
  \end{array} \right.$, that is, $g (1, \alpha, c) = g (1, d, c) \circ g (1,
  \alpha, d) |_{\tmop{Dom} g (n, \alpha, c)}$; finally, direct from the
  previous equality follows that $g^{- 1} (n, \alpha, c) = g^{- 1} (n, \alpha,
  d) \circ g^{- 1} (n, d, c)$ because $g (1, \alpha, c)$, $g (1, d, c)$ and $g
  (1, \alpha, d) |_{\tmop{Dom} g (n, \alpha, c)}$ are invertible functions,
  and since by (2.3.2) $g^{- 1} (1, \alpha, c) = g (1, c, \alpha)$, then \\
  $g^{- 1} (1, \alpha, c) = g (1, c, \alpha) : ( \mathbbm{E} \cap \alpha \cap
  c) \cup \{c\} = \mathbbm{E} \cap c^+ \longrightarrow ( \mathbbm{E} \cap
  \alpha \cap c) \cup \{\alpha\} = \tmop{Dom} g (1, \alpha, c)$.

  {\noindent}(3).\\
  Of course, the relation $\leqslant^1$ from {\cite{GarciaCornejo1}} satisfies
  \\
  $\forall \alpha \in \tmop{Class} (1) . \forall t \in [\alpha, \alpha ( +^1)]
  . \alpha \leqslant^1 t \Longrightarrow \alpha \leqslant_1 t$ and moreover:

  {\noindent}(4) holds because of {\cite{GarciaCornejo1}} proposition
  \ref{<less>^1.implies.cofinal.sequence};

  {\noindent}(5) holds too because of {\cite{GarciaCornejo1}} proposition
  \ref{2nd_Fund_Cof_Property_<less>^1};

  {\noindent}(6) holds because {\cite{GarciaCornejo1}} of corollary
  \ref{<less>_1.iff.<less>^1}.
\end{proof}

\addvspace{4mm}
{\Large \tmstrong{Working on case $n > 1$ of theorem
\ref{most_most_general_theorem}}}

It is in this moment that the hard work starts. As we have already said, we
prove theorem \ref{most_most_general_theorem} by induction on $[1, \omega)$,
and since we have already seen that it holds for $n = 1$, then {\color{blue}
until we complete the whole proof}, we consider a fixed $n \in [2, \omega)$
and our induction hypothesis is that theorem \ref{most_most_general_theorem}
holds for any $i \in [1, n)$. We name {\tmstrong{GenThmIH}} to this induction
hypothesis.

\section{Clause (0) of theorem \ref{most_most_general_theorem}}

We will dedicate the rest of this article to the proof of clause (0) of
theorem \ref{most_most_general_theorem}. (The rest of the proof will appear in
two subsequent articles). In order to do this, our first goal is to provide a
generalized version of the hierarchy theorem done for the intervals
$[\varepsilon_{\gamma}, \varepsilon_{\gamma + 1})$. We first prove certain
propositions that will be necessary later.

\begin{proposition}
  \label{regular_ordinal_in_Class(n-1)}Let $i \in [1, n - 1]$. Let $\kappa$ be
  an uncountable regular ordinal. Then $\kappa \in \tmop{Class} (i)$.
\end{proposition}

\begin{proof}
  Take $i, \kappa$ as stated. Let $\rho$ be an uncountable regular ordinal,
  $\rho > \kappa$ ($\rho$ exists because the class of regular ordinals is
  unbounded in the class of ordinals). Since $\tmop{Class} (i) \cap \kappa$ is
  bounded in $\rho$ and $\tmop{Class} (i)$ is club in $\rho$ by GenThmIH, then
  $\sup (\tmop{Class} (i) \cap \kappa) \in \tmop{Class} (i)$. But
  $\tmop{Class} (i) \cap \kappa$ is unbounded in $\kappa$ (by GenThmIH) and
  therefore $\sup (\tmop{Class} (i) \cap \kappa) = \kappa$. These two
  observations prove $\kappa = \sup (\tmop{Class} (i) \cap \kappa) \in
  \tmop{Class} (i)$.
\end{proof}

\begin{proposition}
  \label{Class(n)_is_closed}For any $i \in [1, n]$, $\tmop{Class} (n)$ is
  closed.
\end{proposition}

\begin{proof}
  For $i \leqslant n - 1$ the claim is clear by GenThmIH. So suppose $i = n$.
  
  Let $\alpha \in \tmop{Lim} \tmop{Class} (n)$. Then there exists a sequence
  $(c_{\xi})_{\xi \in X} \subset \tmop{Class} (n) \cap \alpha$ with
  $c_{\xi}\underset{\text{cof}}{\longhookrightarrow}\alpha$. So $(c_{\xi})_{\xi
  \in X} \subset \tmop{Class} (n - 1)$ and since by (0) of GenThmIH
  $\tmop{Class} (n - 1)$ is club in any non-countable regular ordinal
  $\kappa$, then $\alpha \in \tmop{Class} (n - 1)$.

  Now we want to show that $\forall t \in (\alpha, \alpha ( +^{n - 1})) .
  \alpha <_1 t$. \ {\tmstrong{(*)}}
  
  Let $t \in (\alpha, \alpha ( +^n))$. Since $T (n - 1, \alpha, t)$ is finite
  and $c_{\xi}\underset{\text{cof}}{\longhookrightarrow}\alpha$, then we can
  assume without loss of generality that $\forall \xi \in X.T (n - 1, \alpha,
  t) \cap \alpha \subset c_{\xi}$. This way, for all $\xi \in X$, the ordinal
  \\
  $t [g (n - 1, \alpha, c_{\xi})] \in (c_{\xi}, c_{\xi} ( +^{n - 1}))$ and
  since by hypothesis $c_{\xi} \in \tmop{Class} (n)$, (i.e., $c_{\xi} <_1
  c_{\xi} (+^{n - 1}))$, then $c_{\xi} <_1 t [g (n - 1, \alpha, c_{\xi})]$ by
  $<_1$-connectedness. This shows \\
  $\forall \xi \in X.c_{\xi} <_1 t [g (n - 1, \alpha, c_{\xi})]$.
  
  From our work in the previous paragraph follows that \\
  $\alpha \in \tmop{Lim} \{\gamma \in \tmop{Class} (n - 1) |T (n - 1, \alpha,
  t) \cap \alpha \subset \gamma \wedge \gamma \leqslant_1 t [g (n - 1, \alpha,
  \gamma)]) \}$, and therefore, by use of GenThmIH (5) (Second fundamental
  cofinality property of $\leqslant^{n - 1}$), follows $\alpha \leqslant_1 t$.
  
  The previous shows (*).

  Finally, for the sequence $(d_{\xi})_{\xi \in (\alpha, \alpha ( +^n))}$
  defined as $d_{\xi} \assign \xi$, it follows from (*) that\\
  $\alpha <_1 d_{\xi}\underset{\text{cof}}{\longhookrightarrow}\alpha (+^{n -
  1})$; therefore, by $\leqslant_1$-continuity, $\alpha <_1 \alpha (+^{n -
  1})$, that is, $\alpha \in \tmop{Class} (n)$.
\end{proof}

\begin{remark}
  Consider $i \in [1, n]$, $\alpha \in \tmop{Class} (i)$ and $t \in [\alpha,
  \alpha ( +^i))$. Let $j \in [1, i]$. Then $\lambda (j, e)$ was defined as
  the only one ordinal $\delta$ satisfying $\delta \in \tmop{Class} (j) \wedge
  e \in [\delta, \delta ( +^j))$ or $- \infty$ in case such ordinal does not
  exist. We want {\tmstrong{to show that}} $\tmmathbf{\lambda (j, e)}$
  {\tmstrong{is indeed an ordinal}}:
  
  Let $U \assign (e + 1) \cap \tmop{Class} (j)$. Then $\beta \in U \neq
  \emptyset$ because $j \leqslant i$ implies \\
  $\tmop{Class} (i) \subset \tmop{Class} (j)$ by proposition
  \ref{Class(n)_first_properties}. Let $u \assign \sup U$. Then, by previous
  proposition \ref{Class(n)_is_closed}, \\
  $u \in \tmop{Class} (j) \cap (e + 1)$. Moreover, $e \in [u, u ( +^j))$. This
  shows that $\lambda (e, j) = u \in \tmop{OR}$.
\end{remark}

\begin{proposition}
  \label{Simplest_Characterization_beta_in_Lim(Class(k))}Let $k < n$ and
  $\beta \in \tmop{Class} (k)$. Then $\beta \leqslant_1 \beta (+^{k - 1})
  \ldots ( +^2) ( +^1) 2 + 1 \Longleftrightarrow \beta \leqslant^k \beta (+^{k
  - 1}) \ldots ( +^2) ( +^1) 2 + 1 \Longleftrightarrow \beta \in \tmop{Lim}
  (\tmop{Class} (k))$.
\end{proposition}

\begin{proof}
  Let $k < n$ and $\beta \in \tmop{Class} (k)$.
  
  Note $\beta \leqslant_1 \beta (+^{k - 1}) \ldots ( +^2) ( +^1) 2 + 1
  \Longleftrightarrow \beta \leqslant^k \beta ( +^{k - 1}) \ldots ( +^2) (
  +^1) 2 + 1$ holds because \\
  $\eta (k, \beta, \beta ( +^{k - 1}) \ldots ( +^1) 2) = \beta (+^{k - 1})
  \ldots ( +^1) 2$ and because of (3) and (6) of GenThmIH. Moreover, $\beta
  \leqslant^k \beta (+^{k - 1}) \ldots ( +^1) 2 + 1 \Longrightarrow \beta \in
  \tmop{Lim} (\tmop{Class} (k))$ holds because of (4) of GenThmIH.

  It only remains to show that $\beta \leqslant^k \beta (+^{k - 1}) \ldots (
  +^1) 2 + 1 \Longleftarrow \beta \in \tmop{Lim} (\tmop{Class} (k))$. Take \\
  $\beta \in \tmop{Lim} (\tmop{Class} (k))$. Then there is a sequence
  $(c_{\xi})_{\xi \in X} \subset \tmop{Class} (k)$ with $c_{\xi}
  \underset{\text{cof}}{\longhookrightarrow} \beta$. Now, by (2.1.1) of
  GenThmIH, $T (k, \beta, \beta ( +^{k - 1}) \ldots ( +^1) 2)$ is finite, and
  so $T (k, \beta, \beta ( +^{k - 1}) \ldots ( +^1) 2) \cap \beta$ is finite
  too. This way, there is a subsequence $(d_j)_{j \in J}$ of $(c_{\xi})_{\xi
  \in X}$ such that \\
  $\forall j \in J.T (k, \beta, \beta ( +^{k - 1}) \ldots ( +^1) 2) \cap \beta
  \subset d_j$ and $d_j \underset{\text{cof}}{\longhookrightarrow} \beta$.

  From the previous paragraph we get that $\forall j \in J.T (k, \beta, \beta
  ( +^{k - 1}) \ldots ( +^1) 2) \cap \beta \subset d_j$, $d_j
  \underset{\text{cof}}{\longhookrightarrow} \beta$ and $\forall j \in J.d_j
  \underset{\text{\tmop{by} \tmop{proposition}
  \ref{Class(n)_first_properties}}}{\leqslant_1} d_j (+^{k - 1}) \ldots ( +^1)
  2 \underset{\text{by (2.4.3) and (2.4.1) of GenThmIH}}{=}$\\
  \ \ \ \ \ \ \ \ \ \ \ \ \ \ \ \ \ \ \ \ \ \ \ \ \ \ \ \ \ \ $= (\beta (
  +^{k - 1}) \ldots ( +^1) 2) [g (k, \beta, d_j)]$. That is, we have shown \\
  $\beta \in \tmop{Lim} \{\gamma \in \tmop{Class} (k) |T (k, \beta, \beta (
  +^{k - 1}) \ldots ( +^1) 2) \cap \beta \subset \gamma \wedge$
  
  \ \ \ \ \ \ \ \ \ \ \ \ \ \ \ \ \ \ \ \ \ \ \ \ \ \ \ \ \ \ \ \ \ $\gamma
  \leqslant_1 (\beta ( +^{k - 1}) \ldots ( +^1) 2) [g (k, \beta, \gamma)] \}$.
  Therefore, by (5) of \\
  GenThmIH, we conclude $\beta \leqslant^k \beta (+^{k - 1}) \ldots ( +^1) 2 +
  1$.
\end{proof}

\begin{definition}
  Let $i \in [1, n)$, $\alpha \in \tmop{Class} (i)$ and $t \in (\alpha, \alpha
  ( +^i))$. For any ordinal $r \in \tmop{OR}$, let \\
  $S (i, \alpha, r, t) \assign \{q \in (\alpha, l (i, \alpha, t)) | T (i,
  \alpha, q) \cap \alpha \subset r\}$.
\end{definition}

\begin{remark}
  \label{S(i,a,r,t)=l(i,a,t)_intersection_Dom(g(i,a,r))}With respect to our
  previous definition, note $S (i, \alpha, r, t) \subset l (i, \alpha, t)
  \leqslant t$. Moreover, since $i \in [1, n)$, then by (2.2.3) of GenThmIH,
  \\
  $r \in \tmop{Class} (i) \Longrightarrow S (i, \alpha, r, t) = \{q \in
  (\alpha, l (i, \alpha, t)) | \tmop{Ep} (q) \subset (\tmop{Dom} g (i, \alpha,
  r))\}$.
\end{remark}

\subsection{The Generalized Hierarchy Theorem}

\begin{definition}
  Let $C^{n - 1} : \tmop{OR} \longrightarrow \tmop{Class} (n - 1)$ be the
  counting functional of $\tmop{Class} (n - 1)$, (by GenThmIH follows
  $\tmop{Class} (n - 1)$ is a closed unbounded class of ordinals) and for $j
  \in \tmop{OR}$, let's write $C^{n - 1}_j$ for $C^{n - 1} (j)$.

  We define by recursion on the interval $[C_{\omega}^{n - 1}, \infty)$ the
  functional \\
  $A^{n - 1} : [C^{n - 1}_{\omega}, \infty) \longrightarrow \tmop{Subclasses}
  (\tmop{OR})$ as:

  For $t \in [C^{n - 1}_{\omega}, \infty)$, let $\alpha \in \tmop{Class} (n -
  1)$ be such that $t \in [\alpha, \alpha ( +^{n - 1}))$. \\
  Let $M \assign \left\{ \begin{array}{l}
    \max (T (n - 1, \alpha, t) \cap \alpha) \text{ iff } T (n - 1,
    \alpha, t) \cap \alpha \neq \emptyset\\
    - \infty \text{ otherwise}
  \end{array} \right.$.

  Case $t = l + 1$.

  $A^{n - 1} (l + 1) \assign \left\{ \begin{array}{l}
    A^{n - 1} (l)  \text{ iff } l < \eta (n - 1, \alpha, l)\\
    \\
    \tmop{Lim} A^{n - 1} (l) \text{ otherwise; that is, }
    l = \eta (n - 1, \alpha, l)
  \end{array} \right.$

  Case $t \in \tmop{Lim}$.

  $A^{n - 1} (t) \assign \left\{ \begin{array}{l}
    (\tmop{Lim} \tmop{Class} (n - 1)) \cap (M, \alpha + 1) \text{ iff } t
    \in [\alpha, \alpha ( +^{n - 2}) ( +^{n - 3}) \ldots ( +^2) ( +^1) 2]\\
    \\
    \tmop{Lim} \{r \leqslant \alpha |M < r \in \bigcap_{s \in S (n - 1,
    \alpha, r, t)} A^{n - 1} (s)\} \text{ otherwise}
  \end{array} \right.$.

  On the other hand, we define the functional $G^{n - 1} : [C^{n -
  1}_{\omega}, \infty) \longrightarrow \tmop{Subclasses} (\tmop{OR})$ in the
  \\
  following way:

  For $t \in [C^{n - 1}_{\omega}, \infty)$, let $\alpha \in \tmop{Class} (n -
  1)$ be such that $t \in [\alpha, \alpha ( +^{n - 1}))$ and let\\
  $G^{n - 1} (t) \assign \{\beta \in \tmop{Class} (n - 1) |T (n - 1, \alpha,
  t) \cap \alpha \subset \beta \leqslant \alpha \wedge \beta \leqslant^{n - 1}
  \eta (n - 1, \alpha, t) [g (n - 1, \alpha, \beta)] + 1\}$ 
  \hspace*{13mm} $=$, by GenThmIH (3) and (6),\\
  \hspace*{14mm}$= \{\beta \in \tmop{Class} (n - 1) |T (n - 1, \alpha, t) \cap
  \alpha \subset \beta \leqslant \alpha \wedge \beta \leqslant_1 \eta (n - 1,
  \alpha, t) [g (n - 1, \alpha, \beta)] + 1\}$.
\end{definition}

\begin{remark}
  Notice that $G^{n - 1} (t)$ is well defined because for $\beta \in
  \tmop{Class} (n - 1)$ satisfying \\
  $T (n - 1, \alpha, t) \cap \alpha \subset \beta$, (2.2.3) and (2.2.4) of
  GenThmIH imply $T (n - 1, \alpha, \eta (n - 1, \alpha, t)) \cap \alpha
  \subset \beta$; therefore, again by (2.2.3) of GenThmIH, $\tmop{Ep} (\eta (n
  - 1, \alpha, t)) \subset \tmop{Dom} (g (n - 1, \alpha, \beta))$.
\end{remark}

\begin{proposition}
  \label{A^n-1_constant_in_[alpha,alpha(+)...(+)2]}Let $\alpha \in
  \tmop{Class} (n - 1) \cap [C_{\omega}^{n - 1}, \infty)$. Then \\
  $\forall t \in [\alpha, \alpha ( +^{n - 2}) \ldots ( +^1) 2] .A^{n - 1} (t)
  = (\tmop{Lim} \tmop{Class} (n - 1)) \cap (\max (T (n - 1, \alpha, t) \cap
  \alpha), \alpha + 1)$
\end{proposition}

\begin{proof}
  {\color{orange} Left to the reader.}
\end{proof}

\begin{theorem}
  \label{G^n-1(t)=A^n-1(t)_Gen_Hrchy_thm}$\forall t \in [C^{n - 1}_{\omega},
  \infty) .G^{n - 1} (t) = A^{n - 1} (t)$
\end{theorem}

\begin{proof}
  We proceed by induction on the class $[C^{n - 1}_{\omega}, \infty)$.

  Let $t \in [C^{n - 1}_{\omega}, \infty)$ and $\alpha \in \tmop{Class} (n -
  1)$ be with $t \in [\alpha, \alpha ( +^{n - 1}))$.

  Suppose $\forall s \in [C_{\omega}^{n - 1}, \infty) \cap t.G^{n - 1} (s) =
  A^{n - 1} (s)$. \ \ \ {\tmstrong{(cIH)}}

  Case $t \in [\alpha, \alpha ( +^{n - 2}) ( +^{n - 3}) \ldots ( +^2) ( +^1)
  2]$.
  
  Then $\eta (n - 1, \alpha, t) = \alpha (+^{n - 2}) ( +^{n - 3}) \ldots (
  +^2) ( +^1) 2$ and so
  
  $G^{n - 1} (t) = \{\beta \in \tmop{Class} (n - 1) |T (n - 1, \alpha, t) \cap
  \alpha \subset \beta \leqslant \alpha \wedge$
  
  \ \ \ \ \ \ \ \ \ \ \ \ \ \ $\beta \leqslant^{n - 1} \eta (n - 1, \alpha,
  t) [g (n - 1, \alpha, \beta)] + 1\}=$
  
  \ \ \ \ \ \ \ \ \ \ $=\{\beta \in \tmop{Class} (n - 1) |T (n - 1, \alpha,
  t) \cap \alpha \subset \beta \leqslant \alpha \wedge$
  
  \ \ \ \ \ \ \ \ \ \ \ \ \ \ $\beta \leqslant^{n - 1} \alpha (+^{n - 2}) (
  +^{n - 3}) \ldots ( +^2) ( +^1) 2 [g (n - 1, \alpha, \beta)] + 1\} =$
  
  \ \ \ \ \ \ \ \ \ \ $=\{\beta \in \tmop{Class} (n - 1) |T (n - 1, \alpha,
  t) \cap \alpha \subset \beta \leqslant \alpha \wedge$
  
  \ \ \ \ \ \ \ \ \ \ \ \ \ \ $\beta \leqslant^{n - 1} \beta (+^{n - 2}) (
  +^{n - 3}) \ldots ( +^2) ( +^1) 2 + 1\} =$
  
  \ \ \ \ \ $underset{ proposition
  \ref{Simplest_Characterization_beta_in_Lim(Class(k))}}{=}(\tmop{Lim}
  \tmop{Class} (n - 1)) \cap (\max (T (n - 1, \alpha, t) \cap \alpha), \alpha
  + 1) =$
  
  \ \ \ \ $underset{by proposition
  \ref{A^n-1_constant_in_[alpha,alpha(+)...(+)2]} }{=}A^{n - 1} (t)$.

  The previous shows the theorem holds in interval $[\alpha, \alpha ( +^{n -
  2}) ( +^{n - 3}) \ldots ( +^2) ( +^1) 2]$. So, {\tmstrong{from now on, we
  suppose}} {\tmstrong{$t \in (\alpha ( +^{n - 2}) ( +^{n - 3}) \ldots ( +^2)
  ( +^1) 2, \alpha ( +^{n - 1}))$}}. \ \ \ \ \ \ \ {\tmstrong{(A0)}}

  Successor subcase. Suppose $t = s + 1$ for some $s \in [\alpha ( +^{n - 2})
  ( +^{n - 3}) \ldots ( +^2) ( +^1) 2, t)$.

  First note \\
  $\eta (n, \alpha, s + 1) = \max \{m (e) |e \in (\alpha, s + 1]\} = \max
  \{\max \{m (e) |e \in (\alpha, s]\}, m (s + 1) = s + 1\} =$
  
  \ \ \ \ \ \ \ \ \ \ \ $= \left\{ \begin{array}{l}
    \max \{\max \{m (e) |e \in (\alpha, s]\}, s + 1\} \text{ iff } s =
    \alpha (+^{n - 1}) ( +^{n - 2}) \ldots ( +^2) ( +^1) 2\\
    \\
    \max \{\max \{m (e) |e \in (\alpha, s]\}, s + 1\} \text{ iff } s >
    \alpha (+^{n - 1}) ( +^{n - 2}) \ldots ( +^2) ( +^1) 2
  \end{array} \right.$

  \ \ \ \ \ \ \ \ \ \ \ $= \left\{ \begin{array}{l}
    \max \{\alpha ( +^{n - 1}) ( +^{n - 2}) \ldots ( +^2) ( +^1) 2, s + 1\}
    \text{ iff } s = \alpha (+^{n - 1}) ( +^{n - 2}) \ldots ( +^2) ( +^1)
    2\\
    \\
    \max \{\eta (n - 1, \alpha, s), s + 1\} \text{ iff } s > \alpha (+^{n
    - 1}) ( +^{n - 2}) \ldots ( +^2) ( +^1) 2
  \end{array} \right.$
  
  \ \ \ \ \ \ \ \ \ \ \ $= \left\{ \begin{array}{l}
    \max \{\eta (n - 1, \alpha, s), s + 1\} \text{ iff } s = \alpha (+^{n
    - 1}) ( +^{n - 2}) \ldots ( +^2) ( +^1) 2\\
    \\
    \max \{\eta (n - 1, \alpha, s), s + 1\} \text{ iff } s > \alpha (+^{n
    - 1}) ( +^{n - 2}) \ldots ( +^2) ( +^1) 2
  \end{array} \right.$
  
  \ \ \ \ \ \ \ \ \ \ \ $= \max \{\eta (n - 1, \alpha, s), s + 1\}$. \ \ \ \
  \ \ \ {\tmstrong{(A1)}}

  {\underline{Subsubcase $s < \eta (n - 1, \alpha, s)$}}.
  
  Then, using (A1), $\eta (n - 1, \alpha, s + 1) = \eta (n - 1, \alpha, s)$.
  Therefore, \\
  $G^{n - 1} (t) = G^{n - 1} (s + 1) = \{\beta \in \tmop{Class} (n - 1) |T (n
  - 1, \alpha, s + 1) \cap \alpha \subset \beta \leqslant \alpha \wedge$
  
  \ \ \ \ \ \ \ \ \ \ \ \ \ \ \ \ \ \ \ \ \ \ \ \ \ \ \ \ \ \ \ \ \ $\beta
  \leqslant^{n - 1} \eta (n - 1, \alpha, s + 1) [g (n - 1, \alpha, \beta)] +
  1\}=$
  
  \ \ \ \ \ \ \ \ \ \ \ \ \ \ $underset{\text{by (2.1.2) of GenThmIH}}{=} \{\beta
  \in \tmop{Class} (n - 1) |T (n - 1, \alpha, s) \cap \alpha \subset \beta
  \leqslant \alpha \wedge$
  
  \ \ \ \ \ \ \ \ \ \ \ \ \ \ \ \ \ \ \ \ \ \ \ \ \ \ \ \ \ \ \ \ \ \ \ \ \ \
  \ \ \ \ \ \ \ $\beta \leqslant^{n - 1} \eta (n - 1, \alpha, s) [g (n - 1,
  \alpha, \beta)] + 1\}=$
  
  \ \ \ \ \ \ \ \ \ \ \ \ \ \ \ \ \ \ \ \ \ \ \ \ $= G^{n - 1}
  (s)\underset{\text{by cIH}}{=}A^{n - 1} (s)\underset{\text{because } s < \eta (n - 1,
  \alpha, s)}{=}A^{n - 1} (s + 1)$.

  {\underline{Subsubcase $s = \eta (n - 1, \alpha, s)$}}.
  
  So, from (A1), $\eta (n - 1, \alpha, s + 1) = s + 1 = \eta (n - 1, \alpha,
  s) + 1$. \ \ \ \ \ \ \ {\tmstrong{(A2)}}.

  {\tmstrong{To show $G^{n - 1} (t) \subset A^{n - 1} (t)$.}} \ \ \ \ \ \ \
  {\tmstrong{(A3)}}
  
  Let $\beta \in G^{n - 1} (t) = G^{n - 1} (s + 1)$. Then $\beta \in
  \tmop{Class} (n - 1)$, $T (n - 1, \alpha, s + 1) \cap \alpha \subset \beta
  \leqslant \alpha$ and \\
  $\beta \leqslant^{n - 1} \eta (n - 1, \alpha, s + 1) [g (n - 1, \alpha,
  \beta)] + 1\underset{\text{by (A2)}}{=}(\eta (n - 1, \alpha, s) + 1) [g (n - 1,
  \alpha, \beta)] + 1$; from this and (4) of GenThmIH follows the existence of
  a sequence \\
  $(c_{\xi})_{\xi \in X} \subset \tmop{Class} (n - 1) \cap \beta$,
  $c_{\xi}\underset{\text{cof}}{\longhookrightarrow}\beta$ such that for all $\xi
  \in X$, \\
  $T (n - 1, \beta, {\color{magenta} (\eta (n - 1, \alpha, s) + 1) [g (n - 1,
  \alpha, \beta)]}) \cap \beta \subset c_{\xi}$ and \\
  $c_{\xi} \leqslant_1 \tmmathbf{(\eta (n - 1, \alpha, s) + 1)} [g (n - 1,
  \alpha, \beta)] g [(n - 1, \beta, c_{\xi})]$. \ \ \ \ \ \ \
  {\tmstrong{(A4)}}

  On the other hand, for any $\xi \in X$, $c_{\xi} \supset T (n - 1, \beta,
  {\color{magenta} (\eta (n - 1, \alpha, s) + 1) [g (n - 1, \alpha, \beta)]})
  \cap \beta =$\\
  $T (n - 1, \beta, {\color{magenta} (s + 1) [g (n - 1, \alpha, \beta)]}) \cap
  \beta\underset{\text{by (2.2.4) of GenThmIH}}{=}T (n - 1, \alpha, s + 1) \cap
  \alpha =$\\
  $T (n - 1, \alpha, \eta (n - 1, \alpha, s) + 1) \cap \alpha$. \ \ \ \ \ \ \
  {\tmstrong{(A5)}}

  Now, note that for any $\xi \in X$, by (A5) and (2.2.3) of GenThmIH, we have
  that\\
  $\tmop{Ep} (\eta (n - 1, \alpha, s) + 1) \subset \tmop{Dom} (g (n - 1,
  \alpha, c_{\xi}))$. Then\\
  $(\eta (n - 1, \alpha, s) + 1) [g (n - 1, \alpha, \beta)] [g (n - 1, \beta,
  c_{\xi})] =$\\
  $(\eta (n - 1, \alpha, s) + 1) [g (n - 1, \beta, c_{\xi}) \circ g (n - 1,
  \alpha, \beta)]\underset{\text{by (2.5.3) of GenThmIH}}{=}$\\
  $(\eta (n - 1, \alpha, s) + 1) [g (n - 1, \alpha, c_{\xi})] = \eta (n - 1,
  \alpha, s) [g (n - 1, \alpha, c_{\xi})] + 1\underset{\text{by (2.2.5) of
  GenThmIH}}{=}$\\
  $\eta (n - 1, c_{\xi}, s [g (n - 1, \alpha, c_{\xi})]) + 1$. \ \ \ \ \
  {\tmstrong{(A6)}}

  Done the previous work, from (A4), (A5) (and (2.1.2) of GenThmIH) and (A6)
  follows \\
  $\forall \xi \in X.T (n - 1, \alpha, s) \cap \alpha \subset c_{\xi}
  \leqslant \alpha \wedge c_{\xi} \leqslant_1 \eta (n - 1, \alpha, s [g (n -
  1, \alpha, c_{\xi})]) + 1$ and therefore, by (6) of GenThmIH, $\forall \xi
  \in X.T (n - 1, \alpha, s) \cap \alpha \subset c_{\xi} \leqslant \alpha
  \wedge c_{\xi} \leqslant^{n - 1} \eta (n - 1, \alpha, s [g (n - 1, \alpha,
  c_{\xi})]) + 1$. This shows that $(c_{\xi})_{\xi \in X} \subset G^{n - 1}
  (s)\underset{\text{by our cIH}}{=}A^{n - 1} (s)$, and since
  $c_{\xi}\underset{\text{cof}}{\longhookrightarrow}\beta$, then we have that \\
  $\beta \in \tmop{Lim} A^{n - 1} (s) = A (s + 1) = A (t)$. This proves
  {\tmstrong{(A3)}}.

  {\tmstrong{We now show $G^{n - 1} (t) \supset A^{n - 1} (t)$.}} \ \ \ \ \ \
  \ {\tmstrong{(B1)}}
  
  Let $\beta \in A^{n - 1} (t) = A^{n - 1} (s + 1) = \tmop{Lim} A^{n - 1}
  (s)\underset{\text{by our cIH}}{=}\tmop{Lim} G^{n - 1} (s)$. So there is a
  sequence \\
  $(c_{\xi})_{\xi \in X} \subset G^{n - 1} (s)$ such that
  $c_{\xi}\underset{\text{cof}}{\longhookrightarrow}\beta$. So for all $\xi \in
  X$, \\
  $T (n - 1, \alpha, s) \cap \alpha \subset c_{\xi} \in \tmop{Class} (n - 1)
  \cap \beta \subset \alpha$ and \\
  $c_{\xi} \leqslant^{n - 1} \eta (n - 1, \alpha, s) [g (n - 1, \alpha,
  c_{\xi})] + 1 = (\eta (n - 1, \alpha, s) + 1) [g (n - 1, \alpha, c_{\xi})]$.
  \ \ {\tmstrong{(B2)}}

  We will argue similarly as in the proof of (A3). Let $\xi_0 \in X$. Then \
  \\
  $T (n - 1, \alpha, \eta (n - 1, \alpha, s) + 1) \cap \alpha = T (n - 1,
  \alpha, \eta (n - 1, \alpha, s)) \cap \alpha = T (n - 1, \alpha, s) \cap
  \alpha \subset c_{\xi_0} < \beta$, so $\tmmathbf{T (n - 1, \alpha, s) \cap
  \alpha \subset \beta}$ and the ordinal $(\eta (n - 1, \alpha, s) + 1) [g (n
  - 1, \alpha, \beta)] \in [\beta, \beta ( +^{n - 1}))$ is well defined. Now,
  for any $\xi \in X$, $T (n - 1, \beta, {\color{magenta} (\eta (n - 1,
  \alpha, s) + 1) [g (n - 1, \alpha, \beta)]}) \cap \beta = T (n - 1, \alpha,
  {\color{magenta} (\eta (n - 1, \alpha, s) + 1)}) \cap \alpha = T (n - 1,
  \alpha, s + 1) \cap \alpha = T (n - 1, \alpha, s) \cap \alpha \subset
  c_{\xi}$, and so the ordinal $(\eta (n - 1, \alpha, s) + 1) [g (n - 1,
  \alpha, \beta)] [g (n - 1, \beta, c_{\xi})] \in [c_{\xi}, c_{\xi} ( +^{n -
  1}))$ is well defined too; moreover, using this and (2.5.3) of GenThmIH, we
  get \\
  $(\eta (n - 1, \alpha, s) + 1) [g (n - 1, \alpha, \beta)] [g (n - 1, \beta,
  c_{\xi})] =$\\
  $(\eta (n - 1, \alpha, s) + 1) [g (n - 1, \beta, c_{\xi}) \circ g (n - 1,
  \alpha, \beta)] = (\eta (n - 1, \alpha, s) + 1) [g (n - 1, \alpha,
  c_{\xi})]$. But from this and (B2) we get\\
  $\forall \xi \in X.T (n - 1, \alpha, {\color{magenta} (\eta (n - 1, \alpha,
  s) + 1) [g (n - 1, \alpha, \beta)]}) \cap \beta \subset c_{\xi} < \beta
  \wedge$\\
  $c_{\xi} \leqslant_1 (\eta (n - 1, \alpha, s) + 1) [g (n - 1, \alpha,
  c_{\xi})] = (\eta (n - 1, \alpha, s) + 1) [g (n - 1, \alpha, \beta)] [g (n -
  1, \beta, c_{\xi})]$; note these previous two lines and the fact that
  $c_{\xi}\underset{\text{cof}}{\longhookrightarrow}\beta$ means \\
  $\beta \in \tmop{Lim} \{\gamma \in \tmop{Class} (n - 1) |T (n - 1, \beta,
  {\color{magenta} (\eta (n - 1, \alpha, s) + 1) [g (n - 1, \alpha, \beta)]})
  \cap \beta \subset \gamma \wedge$\\
  \ \ \ \ \ \ \ \ \ \ \ \ \ \ \ \ \ \ \ \ \ \ \ \ \ \ \ \ \ \ \ \ \ \ \
  $\gamma \leqslant_1 {\color{magenta} (\eta (n - 1, \alpha, s) + 1) [g (n -
  1, \alpha, \beta)]} [g (n - 1, \beta, c_{\xi})] \}$. Thus, from all of the
  above and using (5.1) of GenThmIH, we conclude \\
  $T (n - 1, \alpha, s) \cap \alpha \subset \beta \leqslant \alpha \wedge$\\
  $\beta \leqslant^{n - 1} (\eta (n - 1, \alpha, s) + 1) [g (n - 1, \alpha,
  \beta)] + 1\underset{\text{by (A2)}}{=}\eta (n - 1, \alpha, s + 1) [g (n - 1,
  \alpha, \beta)] + 1 =$\\
  \ \ \ \ \ \ \ \ \ \ \ \ \ \ \ \ \ \ \ \ \ \ \ \ \ \ \ \ \ \ \ \ \ \ \ \ \ \
  \ \ \ \ \ \ \ \ \ \ \ \ \ \ \ \ \ \ \ \ \ $= \eta (n - 1, \alpha, t) [g (n -
  1, \alpha, \beta)] + 1$. \ \ \ \ \ \ \ {\tmstrong{(B3)}}.
  
  (B3) shows $\beta \in G^{n - 1}  (t)$. Hence we have shown $G^{n - 1}  (t)
  \supset A^{n - 1} (t)$.

  All the previous work shows that for $t$ a successor ordinal the theorem
  holds. Now we have to see what happens when $t$ is a limit ordinal.

  Subcase $t \in \tmop{Lim}$. We remind the reader that, by (A0), we also know
  that \\
  $t \in (\alpha ( +^{n - 2}) ( +^{n - 3}) \ldots ( +^2) ( +^1) 2, \alpha (
  +^{n - 1}))$.

  {\tmstrong{To show}} $\tmmathbf{G^{n - 1} (t) \subset A^{n - 1} (t)}$. \ \ \
  \ \ \ \ {\tmstrong{(B4)}}
  
  Let $\beta \in G^{n - 1} (t)$. So $T (n - 1, \alpha, t) \cap \alpha \subset
  \beta \leqslant^{n - 1} \eta (n - 1, \alpha, t) [g (n - 1, \alpha, \beta)] +
  1$ and \\
  $\alpha \geqslant \beta \in \tmop{Class} (n - 1)$. Then, by (4) of GenThmIH
  there exists a sequence \\
  $(c_{\xi})_{\xi \in X} \subset \tmop{Class} (n - 1) \cap \beta$, $c_{\xi}
  \underset{\text{cof}}{\longhookrightarrow} \beta$ such that for all $\xi \in
  X$,\\
  $T (n - 1, \beta, \eta (n - 1, \alpha, t)) [g (n - 1, \alpha, \beta)]) \cap
  \beta \subset c_{\xi}$ and \\
  $c_{\xi} \leqslant_1 \tmmathbf{\eta (n - 1, \alpha, t)} [g (n - 1, \alpha,
  \beta)] g [(n - 1, \beta, c_{\xi})]$. \ \ \ \ \ \ \ {\tmstrong{(B5)}}

  On the other hand, since $T (n - 1, \alpha, t) \cap \alpha \subset \beta$,
  $T (n - 1, \alpha, t)$ is finite (by (2.1.1) of GenThmIH) and $c_{\xi}
  \underset{\text{cof}}{\longhookrightarrow} \beta$, then $T (n - 1, \alpha,
  t) \cap \alpha$ is also finite and therefore we can assume without loss of
  generality that $\forall \xi \in X.T (n - 1, \alpha, t) \cap \alpha \subset
  c_{\xi}$. \ \ \ \ \ \ \ {\tmstrong{(B6)}}

  Now, notice for any $\xi \in X$, \\
  $T (n - 1, \alpha, \eta (n - 1, \alpha, t)) \cap \alpha\underset{\text{by (2.1.3)
  of GenThmIH}}{\subset}T (n - 1, \alpha, t) \cap \alpha \subset c_{\xi}$
  and therefore, by (B6) and (2.2.3) of GenThmIH, \\
  $\tmop{Ep} (\eta (n - 1, \alpha, t)) \subset \tmop{Dom} (g (n - 1, \alpha,
  c_{\xi}))$. This way, \\
  $\eta (n - 1, \alpha, t) [g (n - 1, \alpha, \beta)] [g (n - 1, \beta,
  c_{\xi})] =$\\
  $\eta (n - 1, \alpha, t) [g (n - 1, \beta, c_{\xi}) \circ g (n - 1, \alpha,
  \beta)]\underset{\text{(2.5.3) of GenThmIH}}{=}$\\
  $\eta (n - 1, \alpha, t) [g (n - 1, \alpha, c_{\xi})]$. From this, (B5) and
  (B6) we obtain\\
  $\forall \xi \in X.T (n - 1, \alpha, t) \cap \alpha \subset c_{\xi}
  \leqslant \alpha \wedge$\\
  \ \ \ \ \ \ \ \ $c_{\xi} \leqslant_1 \eta (n - 1, \alpha, t) [g (n - 1,
  \alpha, c_{\xi})] = \eta (n - 1, \alpha, l (n - 1, \alpha, t)) [g (n - 1,
  \alpha, c_{\xi})]$. \ \ \ \ \ \ \ {\tmstrong{(C1)}}

  Let's see now that $\forall \xi \in X.c_{\xi} \in\underset{s \in S (n - 1,
  \alpha, c_{\xi}, t)}{\bigcap}A^{n - 1} (s)$. \ \ \ \ \ \ \
  {\tmstrong{(C2)}}
  
  Let $\xi \in X$ be arbitrary. Take $s \in S (n - 1, \alpha, c_{\xi}, t)$.
  Then $s \in (\alpha, l (n - 1, \alpha, t))$ and then, by the definition of
  $l (n - 1, \alpha, t)$, it follows that $\eta (n - 1, \alpha, s) < \eta (n -
  1, \alpha, l (n - 1, \alpha, t))$. On the other hand, since $T (n - 1,
  \alpha, s) \cap \alpha \subset c_{\xi}$ then $T (n - 1, \alpha, \eta (n - 1,
  \alpha, s)) \cap \alpha \subset c_{\xi}$ (by (2.1.3) of GenThmIH); moreover,
  we know $T (n - 1, \alpha, t) \cap \alpha \subset c_{\xi}$, so by (2.1.4) of
  GenThmIH, \\
  $T (n - 1, \alpha, l (n - 1, \alpha, t)) \cap \alpha \subset c_{\xi}$.

  From the previous paragraph follows that, for any $\xi \in X$, the ordinals
  \\
  $\eta (n - 1, \alpha, s) [g (n - 1, \alpha, c_{\xi})], l (n - 1, \alpha, t)
  [g (n - 1, \alpha, c_{\xi})] \in (c_{\xi}, c_{\xi} ( +^{n - 1})) \subset
  \beta < \alpha$ are well defined and that $c_{\xi} < \eta (n - 1, \alpha, s)
  [g (n - 1, \alpha, c_{\xi})] + 1 \leqslant l (n - 1, \alpha, t) [g (n - 1,
  \alpha, c_{\xi})]$. This last inequalities imply, by (C1) and
  $\leqslant_1$-connectedness, that \\
  $c_{\xi} <_1 \eta (n - 1, \alpha, s) [g (n - 1, \alpha, c_{\xi})] +
  1\underset{\text{by (2.2.5) of GenThmIH}}{=}$
  
  \ $\eta (n - 1, c_{\xi}, s [g (n - 1, \alpha, c_{\xi})]) + 1$, and then (by
  (6) of GenThmIH) \\
  $c_{\xi} <^{n - 1} \eta (n - 1, c_{\xi}, s [g (n - 1, \alpha, c_{\xi})]) +
  1\underset{\text{by (2.2.5) of GenThmIH}}{=}\eta (n - 1, \alpha, s) [g (n - 1,
  \alpha, c_{\xi})] + 1$.

  The previous shows that for all $\xi \in X$ and all $s \in S (n - 1, \alpha,
  c_{\xi}, t)$, \\
  $c_{\xi} \in G^{n - 1} (s)\underset{\text{cIH}}{=}A^{n - 1} (s)$, that is, we
  have shown (C2). From (C2) and the fact that
  $c_{\xi}\underset{\text{cof}}{\longhookrightarrow}\beta$ we conclude that
  $\beta \in \tmop{Lim} \{r \leqslant \alpha |M < r \in\underset{s \in S (n
  - 1, \alpha, r, t)}{\bigcap}A^{n - 1} (s) \} = A^{n - 1} (t)$. This
  shows (B4).

  {\tmstrong{Now we show}} $\tmmathbf{G^{n - 1} (t) \supset A^{n - 1} (t)}$. \
  \ \ \ \ \ \ {\tmstrong{(C3)}}
  
  Let $\beta \in A^{n - 1} (t) = \tmop{Lim} \{r \leqslant \alpha |M < r
  \in\underset{s \in S (n - 1, \alpha, r, t)}{\bigcap}A^{n - 1} (s)
  \}\underset{\text{cIH}}{=}$\\
  \ \ \ \ \ \ \ \ \ \ \ \ \ \ \ \ \ \ \ \ \ \ \ \ $= \tmop{Lim} \{r \leqslant
  \alpha |M < r \in\underset{s \in S (n - 1, \alpha, r, t)}{\bigcap}G^{n
  - 1} (s) \}$. Then there is a sequence\\
  $(c_{\xi})_{\xi \in X}$ such that $M < c_{\xi}
  \underset{\text{cof}}{\longhookrightarrow} \beta$ and $\forall \xi \in
  X.c_{\xi} \in\underset{s \in S (n - 1, \alpha, c_{\xi},
  t)}{\bigcap}G^{n - 1} (s)$. \ \ \ \ \ \ \ {\tmstrong{(C4)}} \\
  Note that since $\forall \xi \in X.c_{\xi} \in\underset{s \in S (n - 1,
  \alpha, c_{\xi}, t)}{\bigcap}G^{n - 1} (s) \subset \tmop{Class} (n - 1)$
  and $(c_{\xi})_{\xi \in X}$ is cofinal in $\beta$, then, by proposition
  \ref{Class(n)_is_closed}, $\beta \in \tmop{Class} (n - 1)$.

  Now, for any $\xi \in X$, we know $\max (T (n - 1, \alpha, t) \cap \alpha)
  = M < c_{\xi} < \beta$; therefore, by (2.2.3) and (2.1.4) of GenThmIH, we
  have hat\\
  $t [g (n - 1, \alpha, \beta)], l (n - 1, \alpha, t) [g (n - 1, \alpha,
  \beta)] \in (\beta, \beta ( +^{n - 1}))$ and $t [g (n - 1, \alpha,
  c_{\xi})], l (n - 1, \alpha, t) [g (n - 1, \alpha, c_{\xi})] \in (c_{\xi},
  c_{\xi} ( +^{n - 1}))$ are well defined. \ \ \ \ \ \ \ {\tmstrong{(C5)}}

  Our next aim is to show that $\forall \xi \in X.c_{\xi} \leqslant_1 l (n -
  1, \alpha, t) [g (n - 1, \alpha, c_{\xi})]$. \ \ \ \ \ \ \ {\tmstrong{(C6)}}

  Let $\xi \in X$ be arbitrary. First note that, since $t \in \tmop{Lim}$,
  then $l (n - 1, \alpha, t) \in \tmop{Lim}$ (because \\
  $l (n - 1, \alpha, t) = t \in \tmop{Lim}$ or $l (n - 1, \alpha, t) < l (n -
  1, \alpha, t) + 1 \leqslant t \leqslant m (t) = m (l (n - 1, \alpha, t))$;
  the latter case implies $l (n - 1, \alpha, t) <_1 l (n - 1, \alpha, t) + 1$
  by $\leqslant_1$-connectedness and so $l (n - 1, \alpha, t) \in
  \mathbbm{P}$) and then $l (n - 1, \alpha, t) [g (n - 1, \alpha, c_{\xi})]
  \in \tmop{Lim}$ (simply because $l (n - 1, \alpha, t) [g (n - 1, \alpha,
  c_{\xi})]$ is the result of substituting epsilon numbers by other epsilon
  numbers in the cantor normal form of $l (n - 1, \alpha, t)$). Now, let $q
  \in (c_{\xi}, l (n - 1, \alpha, t) [g (n - 1, \alpha, c_{\xi})])
  \underset{\text{by (2.2) of GenThmIH}}{\subset}
  (c_{\xi}, c_{\xi} ( +^n))$ be arbitrary. Then by (2.3.1) of GenThmIH,
  $\tmop{Ep} (q) \subset \tmop{Dom} (g (n - 1, c_{\xi}, \alpha))$ and then\\
  $q [g (n - 1, c_{\xi}, \alpha)] \underset{\text{by (2.4.3) of GenThmIH}}{\in}
  (\alpha, l (n - 1, \alpha, t) [g (n - 1, \alpha,
  c_{\xi})] [g (n - 1, c_{\xi}, \alpha)]) =$ \\
  $\underset{\text{by (2.3.2) of GenThmIH}}{=}
  (\alpha, l (n - 1, \alpha, t))$. This shows \\
  $q [g (n - 1, c_{\xi}, \alpha)] \in \tmop{Im} (g (n - 1, c_{\xi}, \alpha))
  \cap (\alpha, l (n - 1, \alpha, t))\underset{\text{by (2.3.2) of GenThmIH}}{=}$\\
  $(\tmop{Dom} (g (n - 1, \alpha, c_{\xi}))) \cap (\alpha, l (n - 1, \alpha,
  t))\underset{\text{remark}
  \ref{S(i,a,r,t)=l(i,a,t)_intersection_Dom(g(i,a,r))}}{=}S (n - 1, \alpha,
  c_{\xi}, t)$, and so by (C4), \\
  $c_{\xi} \in G^{n - 1} (q [g (n - 1, c_{\xi}, \alpha)])$. Finally, observe
  the latter implies that \\
  $c_{\xi} \leqslant^{n - 1} \eta (n - 1, \alpha, q [g (n - 1, c_{\xi},
  \alpha)]) [g (n - 1, \alpha, c_{\xi})] + 1\underset{\text{(2.2.5) of
  GenThmIH}}{=}$\\
  \ \ \ \ \ \ $= \eta (n - 1, c_{\xi}, q [g (n - 1, c_{\xi}, \alpha)] [g (n -
  1, \alpha, c_{\xi})]) + 1 = \eta (n - 1, c_{\xi}, q) + 1$, which
  subsequently implies, (using $c_{\xi} < q \leqslant \eta (n - 1, c_{\xi},
  q)$ and $\leqslant_1$-connectedness) that $c_{\xi} \leqslant_1 q$.

  Last paragraph proves that, for $\xi \in X$, the sequence $(d_q)_{q \in Y}$
  defined as $d_q \assign q$ and \\
  $Y \assign (c_{\xi}, l (n - 1, \alpha, t) [g (n - 1, \alpha, c_{\xi})])$,
  satisfies $\forall q \in Y.c_{\xi} \leqslant_1 q$; but this and the fact
  that $d_q\underset{\text{cof}}{\longhookrightarrow}l (n - 1, \alpha, t) [g (n -
  1, \alpha, c_{\xi})]$ (we already showed $l (n - 1, \alpha, t) [g (n - 1,
  \alpha, c_{\xi})] \in \tmop{Lim}$) imply $c_{\xi} \leqslant_1 l (n - 1,
  \alpha, t) [g (n - 1, \alpha, c_{\xi})]$ by $\leqslant_1$-continuity. Since
  the previous was done for arbitrary $\xi \in X$, we conclude $\forall \xi
  \in X.c_{\xi} \leqslant_1 l (n - 1, \alpha, t) [g (n - 1, \alpha,
  c_{\xi})]$. This proves (C6).

  We continue with the proof of (C3).
  
  Let $\xi \in X$. Using (C6) we get \\
  $c_{\xi} \leqslant_1 l (n - 1, \alpha, t) [g (n - 1, \alpha, c_{\xi})]
  \leqslant_1 \eta (n - 1, c_{\xi}, l (n - 1, \alpha, t) [g (n - 1, \alpha,
  c_{\xi})]) =$\\
  \ \ \ \ \ \ \ \ \ \ \ \ \ \ \ \ \ \ \ \ \ \ \ \ \ \ \ \ \ \ \ \ $\underset{\text{by
  (2.5.3) of GenThmIH}}{=}$\\
  \ \ \ \ \ \ \ \ \ \ \ \ \ \ \ \ \ \ \ \ \ \ \ \ \ \ \ \ \ \ \ \ \ \ \ \ \ \
  \ \ \ \ $= \eta (n - 1, c_{\xi}, l (n - 1, \alpha, t) [g (n - 1, \beta,
  c_{\xi}) \circ g (n - 1, \alpha, \beta)]) =$\\
  \ \ \ \ \ \ \ \ \ \ \ \ \ \ \ \ \ \ \ \ \ \ \ \ \ \ \ \ \ \ \ \ \ \ \ \ \ \
  \ \ \ \ $= \eta (n - 1, c_{\xi}, l (n - 1, \alpha, t) [g (n - 1, \alpha,
  \beta)] [g (n - 1, \beta, c_{\xi})]) =$\\
  \ \ \ \ \ \ \ \ \ \ \ \ \ \ \ \ \ \ \ \ \ \ \ \ \ \ \ \ \ \ \ \ $\underset{\text{by
  (2.2.5) of GenThmIH}}{=}$\\
  \ \ \ \ \ \ \ \ \ \ \ \ \ \ \ \ \ \ \ \ \ \ \ \ \ \ \ \ \ \ \ \ \ \ \ \ \ \
  \ \ \ \ $= \eta (n - 1, \beta, l (n - 1, \alpha, t) [g (n - 1, \alpha,
  \beta)]) [g (n - 1, \beta, c_{\xi})]$;\\
  therefore, by $\leqslant_1$-transitivity, $c_{\xi} \leqslant_1 \eta (n - 1,
  \beta, l (n - 1, \alpha, t) [g (n - 1, \alpha, \beta)]) [g (n - 1, \beta,
  c_{\xi})]$. But since this was done for arbitrary $\xi \in X$, we have
  proved \\
  $\forall \xi \in X.c_{\xi} \leqslant_1 \eta (n - 1, \beta, l (n - 1, \alpha,
  t) [g (n - 1, \alpha, \beta)]) [g (n - 1, \beta, c_{\xi})]$. \ \ \ \ \ \ \
  {\tmstrong{(C7)}}

  Finally, from (C7), the fact that $c_{\xi}
  \underset{\text{cof}}{\longhookrightarrow} \beta$ and (5) of GenThmIH follow
  that\\
  $\beta \leqslant^{n - 1} \eta (n - 1, \beta, l (n - 1, \alpha, t) [g (n - 1,
  \alpha, \beta)]) + 1\underset{\text{by (2.2.5) of GenThmIH}}{=}$\\
  \ $= \eta (n - 1, \alpha, l (n - 1, \alpha, t)) [g (n - 1, \alpha, \beta)]
  + 1\underset{\text{by proposition }
  \ref{eta(i,a,l(i,a,t))=eta(i,a,t)_if_t_in_(a(+)...(+^1)2,a(+^i)}}{=}$\\
  \ $= \eta (n - 1, \alpha, t) [g (n - 1, \alpha, \beta)] + 1$. This and (C5)
  show that $\beta \in G^{n - 1} (t)$. But the previous we have done for
  arbitrary $\beta \in A^{n - 1} (t)$, so we have proved $A^{n - 1} (t)
  \subset G^{n - 1} (t)$, i.e., we have proved (C3).
\end{proof}

\subsection{Uncountable regular ordinals and the $A^{n - 1} (t)$ sets}

\begin{proposition}
  \label{A^n-1(t)_club_in_kapa}Let $\kappa$ be an uncountable regular ordinal
  ($\kappa \in \tmop{Class} (n - 1)$ by proposition
  \ref{regular_ordinal_in_Class(n-1)}). Then $\forall t \in [\kappa, \kappa (
  +^{n - 1}))$, $A^{n - 1} (t)$ is club in $\kappa$.
\end{proposition}

\begin{proof}
  We prove the claim by induction on the interval $[\kappa, \kappa ( +^{n -
  1}))$.

  {\tmstrong{Case $t = \kappa$.}}
  
  Then $T (n - 1, \kappa, t) \cap \kappa \underset{\text{definition of } T (n - 1, \kappa, \kappa)}{=} \emptyset$. So \\
  $A^{n - 1} (t) = (\tmop{Lim} \tmop{Class} (n - 1)) \cap (- \infty, \kappa +
  1) = \tmop{Lim} \tmop{Class} (n - 1)$ is club in $\kappa$ because \\
  $\tmop{Class} (n - 1)$ is club in $\kappa$ (by GenThmIH (0)) and because of
  {\cite{GarciaCornejo1}} proposition \ref{X_club_implies_LimX_club}.

  Our induction hypothesis is
  
  $\forall s \in [\kappa, \kappa ( +^{n - 1})) \cap t.A^{n - 1} (s)$ is club
  in $\kappa$. \ \ \ \ \ \ \ {\tmstrong{(IH)}}

  {\tmstrong{Case $t = l + 1 \in [\kappa, \kappa ( +^{n - 1}))$.}}
  
  Then $A^{n - 1} (t) = A^{n - 1} (l + 1) = \left\{ \begin{array}{l}
    A^{n - 1} (l) \text{ if } l < \eta (n - 1, \kappa, l)\\
    \\
    \tmop{Lim} A^{n - 1} (l) \text{ otherwise}
  \end{array} \right.$; this way, by our (IH) and {\cite{GarciaCornejo1}}
  proposition \ref{X_club_implies_LimX_club}, $A^{n - 1} (t)$ is club in
  $\kappa$.

  {\tmstrong{Case $t \in [\kappa, \kappa ( +^{n - 1})) \cap \tmop{Lim}$.}}
  
  By definition

  $M \assign \left\{ \begin{array}{l}
    \max (T (n - 1, \kappa, t) \cap \kappa) \text{ iff } T (n - 1,
    \kappa, t) \cap \kappa \neq \emptyset\\
    - \infty \text{ otherwise}
  \end{array} \right.$

  and

  $A^{n - 1} (t) = \left\{ \begin{array}{l}
    (\tmop{Lim} \tmop{Class} (n - 1)) \cap (M, \kappa + 1) \text{ iff } t
    \in [\kappa, \kappa ( +^{n - 2}) \ldots ( +^2) ( +^1) 2]\\
    \\
    \tmop{Lim} \{r \leqslant \kappa |M < r \in \bigcap_{s \in S (n - 1,
    \kappa, r, t)} A^{n - 1} (s)\}  \text{ otherwise}
  \end{array} \right.$

  If $t \in [\kappa, \kappa ( +^{n - 2}) \ldots ( +^2) ( +^1) 2]$, then \\
  $A^{n - 1} (t) = (\tmop{Lim} \tmop{Class} (n - 1)) \cap (M, \kappa + 1)$ is
  club in $\kappa$ because of exactly the same reasons as in the case $t =
  \kappa$.

  So from now on we suppose $t \in (\kappa ( +^{n - 2}) \ldots ( +^2) ( +^1)
  2, \kappa ( +^{n - 1}))$.

  First we make the following four observations:

  - It is enough to show that $Y \assign \{r \leqslant \kappa |M < r \in
  \bigcap_{s \in S (n - 1, \kappa, r, t)} A^{n - 1} (s)\}$ is club in $\kappa$
  because, knowing this, we conclude Lim$Y = A^{n - 1} (t)$ is club in
  $\kappa$ by {\cite{GarciaCornejo1}} proposition
  \ref{X_club_implies_LimX_club}. Moreover, note that as a consequence of
  theorem \ref{G^n-1(t)=A^n-1(t)_Gen_Hrchy_thm}, $\forall z \in \tmop{Dom}
  A^{n - 1} .A^{n - 1} (z) \subset \tmop{Class} (n - 1)$ and therefore $Y =
  \{r \in \tmop{Class} (n - 1) \cap (\kappa + 1) |M < r \in \bigcap_{s \in S
  (n - 1, \kappa, r, t)} A^{n - 1} (s)\}$. \ \ \ \ \ \ \ {\tmstrong{(0*)}}

  - For $r \in \tmop{Class} (n - 1) \cap \kappa$, \\
  $\{q \in (\kappa, l (n - 1, \kappa, t)) | \tmop{Ep} (q) \subset (\tmop{Im} g
  (n - 1, r, \kappa))\} \underset{\text{by (2.3.2) of GenThmIH}}{=}$\\
  $\{q \in (\kappa, l (n - 1, \kappa, t)) | \tmop{Ep} (q) \subset (\tmop{Dom}
  g (n - 1, \kappa, r))\}\underset{\text{by remark }
  \ref{S(i,a,r,t)=l(i,a,t)_intersection_Dom(g(i,a,r))}}{=}$\\
  $S (n - 1, \kappa, r, t) \underset{\text{by remark } \ref{S(i,a,r,t)=l(i,a,t)_intersection_Dom(g(i,a,r))}}{\subset} l (n - 1,
  \kappa, t) \underset{\text{by remark }\ref{S(i,a,r,t)=l(i,a,t)_intersection_Dom(g(i,a,r))}}{\leqslant} t$. \ \ \
  \ \ \ \ {\tmstrong{(1*)}}

  - By (1*) and our (IH), \\
  $\forall r \in \tmop{Class} (n - 1) \cap \kappa \forall s \in S (n - 1,
  \kappa, r, t)$, $A^{n - 1} (s)$ is club in $\kappa$. \ \ {\tmstrong{(2*)}}

  - Let $r \in \tmop{Class} (n - 1) \cap \kappa$. By (0) of GenThmIH,
  $\tmop{Class} (n - 1)$ is club in $\kappa$ and consequently $r (+^{n - 1})
  \in \tmop{Class} (n - 1) \cap \kappa$; moreover, by proposition
  \ref{regular_ordinal_in_Class(n-1)}, $\kappa \in \tmop{Class} (n - 1)$ and
  subsequently, \\
  $r < r (+^{n - 1}) < \kappa < \kappa (+^{n - 1}$). Consider the function
  $P_r : r (+^{n - 1}) \longrightarrow \kappa (+^{n - 1}$) defined as $P_r (x)
  \assign x [g (n - 1, r, \kappa)]$. $P_r$ is well defined because of (2.3.1)
  of GenThmIH. We now show that $S (n - 1, \kappa, r, t) \subset \tmop{Im}
  P_r$. This is easy: Take $q \in S (n - 1, \kappa, r, t)$. Then, by (1*), \\
  $\tmop{Ep} (q) \subset \tmop{Dom} (g (n - 1, \kappa, r))$ and therefore $q
  [g (n - 1, \kappa, r)]$ is well defined; but then, by (2.3.3) and (2.3.2) of
  GenThmIH, $q [g (n - 1, \kappa, r)] \in r (+^{n - 1}$) and $q = q [g (n - 1,
  \kappa, r)] [g (n - 1, r, \kappa)] = P_r (q [g (n - 1, \kappa, r)])$. This
  shows $S (n - 1, \kappa, r, t) \subset \tmop{Im} P_r$ as we assured.
  Finally, since $P_r$ is a strictly increasing function (so it is injective),
  then \\
  $|S (n - 1, \kappa, r, t) | \leqslant | \tmop{Im} P_r | = |r ( +^{n - 1}) |
  \underset{\text{because } \kappa \text{ is a cardinal}}{<} \kappa$. \ \ \ \ \ \ \ {\tmstrong{(3*)}}

  After the previous observations, we continue with the proof of the theorem,
  that is, as already said in (0*), we want to show that $Y$ is club in
  $\kappa$.

  {\tmstrong{We show first that $Y$ is $\kappa$-closed.}}
  
  Let $(r'_i)_{i \in I'} \subset Y \cap \kappa$ be such that $|I' | < \kappa$
  and $r'_i\underset{\text{cof}}{\longhookrightarrow}\rho$ for some $\rho <
  \kappa$. To show that $\rho \in Y$.
  
  Since $Y \subset \tmop{Class} (n - 1)$ and by (0) of GenThmIH $\tmop{Class}
  (n - 1)$ is club in $\kappa$, then \\
  $\rho \in \tmop{Class} (n - 1)$. Now consider $s \in S (n - 1, \kappa, \rho,
  t) =$\\
  $\{d \in (\kappa, l (n - 1, \kappa, t)) \subset (\kappa, \kappa (+^{n - 1}))
  | T (n - 1, \kappa, d) \cap \kappa \subset \rho\}$. Since by (2.1.1) of
  GenThmIH $T (n - 1, \kappa, s) \cap \kappa$ is finite and
  $r'_i\underset{\text{cof}}{\longhookrightarrow}\rho$, then there exists a
  subsequence $(r_i)_{i \in I}$ of the sequence $(r'_i)_{i \in I'}$, such that
  $r_i\underset{\text{cof}}{\longhookrightarrow}\rho$, $\forall i \in I.T (n - 1,
  \kappa, s) \cap \kappa \subset r_i$ and $|I| \leqslant |I' | < \kappa$; that
  is, \\
  $\forall i \in I.s \in S (n - 1, \kappa, r_i, t)$. This and the fact that
  $(r_i)_{i \in I} \subset Y$ means $\forall i \in I.r_i \in A^{n - 1} (s)$.
  But by (2*) $A^{n - 1} (s)$ is club in $\kappa$, so $\rho = \sup \{r_i |i
  \in I\} \in A^{n - 1} (s)$. Our previous work shows that, for arbitrary $s
  \in S (n - 1, \kappa, \rho, t)$, $\rho \in A^{n - 1} (s)$, i.e., $\rho \in
  \bigcap_{s \in S (n - 1, \kappa, \rho, t)} A^{n - 1} (s)$. From this it
  follows that $\rho \in Y$. Hence $Y$ is $\kappa$-closed.

  {\tmstrong{Now our aim is to to prove that $Y$ is unbounded in $\kappa$.}} \
  \ \ \ \ \ \ {\tmstrong{(b0)}}
  
  We do first the following:

  Let $R \assign \tmop{Class} (n - 1) \cap \kappa$ and $B_r \assign \bigcap_{s
  \in S (n - 1, \kappa, r, t)} A^{n - 1} (s)$ for any $r \in R$.

  Let's show first that $\forall \xi \in \tmop{Lim} R \cap \kappa$.$\bigcap_{r
  \in R \cap \xi} B_r = B_{\xi}$. \ \ \ \ \ \ \ {\tmstrong{(b1)}}
  
  Proof of (b1):
  
  Let $\xi \in \tmop{Lim} R \cap \kappa$.

  {\underline{We show (b1) contention $'' \subset''$.}}
  
  Let $x \in \bigcap_{r \in R \cap \xi} B_r = \bigcap_{r \in R \cap \xi}  (
  \bigcap_{s \in S (n - 1, \kappa, r, t)} A^{n - 1} (s))$ be arbitrary. This
  means \\
  $\forall r \in R \cap \xi . \forall s \in S (n - 1, \kappa, r, t) .x \in
  A^{n - 1} (s)$. \ \ \ \ \ \ \ {\tmstrong{(b2)}}.
  
  On the other hand, let $z \in S (n - 1, \kappa, \xi, t)$ be arbitrary. By
  definition of $S (n - 1, \kappa, \xi, t)$, this means $z \in (\kappa, l (n -
  1, \kappa, t))$ and $T (n - 1, \kappa, z) \cap \kappa \subset \xi$. But $T
  (n - 1, \kappa, d) \cap \kappa$ is a finite set (by (2.1.1) of GenThmIH),
  so, since $\xi \in \tmop{Lim} R$, there exists $r \in R$ such that $T (n -
  1, \kappa, z) \cap \kappa \subset r < \xi$. This means $z \in S (n - 1,
  \kappa, r, t)$, and then, by (b2), $x \in A^{n - 1} (z)$. Note the previous
  shows $\forall z \in S (n - 1, \kappa, \xi, t) .x \in A^{n - 1} (z)$, i.e.,
  $x \in \bigcap_{s \in S (n - 1, \kappa, \xi, t)} A^{n - 1} (s) = B_{\xi}$.
  Finally, since this was done for arbitrary $x \in \bigcap_{r \in R \cap \xi}
  B_r$, then we have actually shown that $\bigcap_{r \in R \cap \xi} B_r
  \subset B_{\xi}$.

  {\underline{Now we show (b1) contention $'' \supset''$.}}
  
  Let $x \in B_{\xi} = \bigcap_{s \in S (n - 1, \kappa, \xi, t)} A^{n - 1}
  (s)$ be arbitrary. This means \\
  $\forall s \in S (n - 1, \kappa, \xi, t) .x \in A^{n - 1} (s)$. \ \ \ \ \
  {\tmstrong{(b4)}}

  On the other hand, let $r \in R \cap \xi$ be arbitrary. Take $z \in S (n -
  1, \kappa, r, t)$. By definition, this means $z \in (\kappa, l (n - 1,
  \kappa, t))$ and $T (n - 1, \kappa, z) \cap \kappa \subset r$. But since $r
  < \xi$, this implies that actually $z \in S (n - 1, \kappa, \xi, t)$, which,
  together with (b4), implies $x \in A^{n - 1} (z)$. Note we have shown \\
  $\forall z \in S (n - 1, \kappa, r, t) .x \in A^{n - 1} (z)$, i.e, $x \in
  \bigcap_{z \in S (n - 1, \kappa, r, t)} A^{n - 1} (z) = B_r$; moreover, we
  have shown this for arbitrary $r \in R \cap \xi$, i.e., we have shown $x \in
  \bigcap_{r \in R \cap \xi} B_r$. Finally, since this was done for arbitrary
  $x \in B_{\xi}$, we have shown $\bigcap_{r \in R \cap \xi} B_r \supset
  B_{\xi}$.
  
  This concludes the proof of (b1).

  Now we show that $X \assign \{r \in R | M < r \in B_r \}$ is unbounded in
  $\kappa$. \ \ \ \ \ \ \ {\tmstrong{(c0)}}.

  By (2*), (3*) and {\cite{GarciaCornejo1}} proposition
  \ref{Intersection_club_classes} we have that for any $r \in R$, $B_r$ is
  club in $\kappa$. \ \ \ \ \ \ \ {\tmstrong{(c1)}}

  Let $\delta \in \kappa$ be arbitrary. Moreover, let $a \assign \min R$. We
  define by recursion the function \\
  $r : \omega \longrightarrow R$ as:
  
  $r (0) \assign \min \{s \in \kappa \cap B_a | \delta < s > M\}$. Note $r
  (0)$ exists because of (c1).
  
  Suppose we have defined $r (l) \in R = \tmop{Class} (n - 1) \cap \kappa$,
  for $l \in \omega$. \ \ \ \ \ \ \ {\tmstrong{(rIH)}}
  
  Note that $|R \cap r (l) | \leqslant r (l) \underset{\text{by (rIH)}}{<}\kappa$, 
  and then, by (c1) and {\cite{GarciaCornejo1}}
  proposition \ref{Intersection_club_classes} it follows that $\underset{z \in
  R \cap r (l)}{\bigcap}B_z$ is club in $\kappa$. So we define $r (l + 1)
  \assign \min \{s \in \kappa \cap\underset{z \in R \cap r
  (l)}{\bigcap}B_z | r (l) < s\}$.

  Consider $\rho \assign \sup \{r (l) | l \in \omega\}$.
  
  First note that, by construction, $(r (l))_{l \in \omega}$ is a strictly
  increasing sequence of ordinals in $R$ (because any $B_s$ is club in
  $\kappa$ and $B_s \subset \tmop{Class} (n - 1)$) and so $\rho \in
  \tmop{Class} (n - 1) \cap (\tmop{Lim} R)$. Moreover, since $\kappa$ is an
  uncountable regular ordinal and $r : \omega \longrightarrow \kappa$, then
  $\rho < \kappa$. Summarizing all these observations: $\rho \in R \cap
  (\tmop{Lim} R)$ \ \ \ \ \ \ {\tmstrong{(c2)}}

  Now we show $M < \rho \wedge \delta < \rho \in B_{\rho}$. \ \ \ \ \ \ \
  {\tmstrong{(c3)}}

  That $\delta < \rho > M$ is clear from the definition of the function $r$.
  Now, let $\gamma \in R \cap \rho$ be arbitrary. Then there exists $l \in
  \omega$ such that $r (l) > \gamma$. Now, by the definition of our function
  $r$, \\
  $r (l + 1) \in\underset{z \in R \cap r (l)}{\bigcap}B_z$; but this
  implies that the sequence $(r (s))_{s \in [l + 1, \omega)} \subset
  B_{\gamma}$, and since \\
  $\rho = \sup \{r (s) | s \in [l + 1, \omega)\}$ and $B_{\gamma}$ is club in
  $\kappa$, then $\rho \in B_{\gamma}$. Finally, since this was done for
  arbitrary $\gamma \in R \cap \rho$, then we have actually shown that $\rho
  \in\underset{\gamma \in R \cap \rho}{\bigcap}B_{\gamma}\underset{\text{by
  (b1) and (c2)}}{=}B_{\rho}$. This concludes the proof of (c3).

  Finally, observe that (c2) and (c3) have actually shown that $\forall \delta
  \in \kappa \exists \rho \in R. \delta < \rho \in X \subset R \subset
  \kappa$. Therefore (c0) holds. But $X \underset{\text{by (0)}}{=} Y
  \cap \kappa \subset Y$. So $Y$ is unbounded in $\kappa$. This has proven
  (b0).
\end{proof}

\begin{proposition}
  \label{alpha<less>^n-1alpha(+^n-1)_iff_alpha_in_intersection}
  
  $\forall \alpha \in \tmop{Class} (n - 1)$.$\alpha <_1 \alpha (+^{n - 1})
  \Longleftrightarrow \alpha <^{n - 1} \alpha (+^{n - 1})$ \\ 
  \hspace*{79mm}$\Longleftrightarrow
  \alpha \in \bigcap_{t \in [\alpha, \alpha ( +^{n - 1}))} A^{n - 1} (t)$.
\end{proposition}

\begin{proof}
  Let $\alpha \in \tmop{Class} (n - 1)$.
  
  To show $\alpha <^{n - 1} \alpha (+^{n - 1}) \Longrightarrow \alpha \in
  \bigcap_{t \in [\alpha, \alpha ( +^{n - 1}))} A^{n - 1} (t)$. \ \ \ \ \ \ \
  {\tmstrong{(a)}}
  
  Suppose $\alpha <^{n - 1} \alpha (+^{n - 1})$. Let $t \in [\alpha, \alpha (
  +^{n - 1}))$ be arbitrary. \\
  Then $\alpha \leqslant^{n - 1} \eta (n - 1, \alpha, t) [g (n - 1, \alpha,
  \alpha)] + 1 = \eta (n - 1, \alpha, t) + 1$ by $<^{n - 1}$-connectedness. So
  \ $\alpha \in \{\beta \in \tmop{Class} (n - 1) |T (n - 1, \alpha, t) \cap
  \alpha \subset \beta \leqslant \alpha \wedge \beta \leqslant^{n - 1} \eta (n
  - 1, \alpha, t) [g (n - 1, \alpha, \beta)] + 1\} =$\\
  $G^{n - 1} (t)\underset{\text{theorem }
  \ref{G^n-1(t)=A^n-1(t)_Gen_Hrchy_thm}}{=}A^{n - 1} (t)$. Since this holds
  for an arbitrary $t \in [\alpha, \alpha ( +^{n - 1}))$, we have shown
  $\alpha \in \bigcap_{t \in [\alpha, \alpha ( +^{n - 1}))} A^{n - 1} (t)$.
  This shows (a).

  To show $\alpha <^{n - 1} \alpha (+^{n - 1}) \Longleftarrow \alpha \in
  \bigcap_{t \in [\alpha, \alpha ( +^{n - 1}))} A^{n - 1} (t)$. \ \ \ \ \ \ \
  {\tmstrong{(b)}}
  
  Suppose $\alpha \in \bigcap_{t \in [\alpha, \alpha ( +^{n - 1}))} A^{n - 1}
  (t)\underset{\text{theorem } \ref{G^n-1(t)=A^n-1(t)_Gen_Hrchy_thm}}{=}\bigcap_{t
  \in [\alpha, \alpha ( +^{n - 1}))} G^{n - 1} (t)$. Then for any \\
  $t \in [\alpha, \alpha ( +^{n - 1}))$, $\alpha \leqslant^{n - 1} \eta (n -
  1, \alpha, t) [g (n - 1, \alpha, \alpha)] + 1 = \eta (n - 1, \alpha, t) +
  1$; thus, by (3) of GenThmIH (that is, by $\leqslant^{n - 1}$-continuity),
  $\alpha \leqslant^{n - 1} \alpha (+^{n - 1}$). This shows (b).

  To show $\alpha <_1 \alpha (+^{n - 1}) \Longrightarrow \alpha <^{n - 1}
  \alpha (+^{n - 1})$. \ \ \ \ \ \ \ {\tmstrong{(c)}}
  
  Suppose $\alpha <_1 \alpha (+^{n - 1})$. Then for any $t \in [\alpha, \alpha
  ( +^{n - 1}))$, $\eta (n - 1, \alpha, t)) + 1 \in (\alpha, \alpha ( +^{n -
  1}))$ and so, by $\leqslant_1$-connectedness, $\alpha \leqslant_1 \eta (n -
  1, \alpha, t)) + 1$. Subsequently, by (6) of GenThmIH, $\alpha \leqslant^{n
  - 1} \eta (n - 1, \alpha, t)) + 1$. The previous shows that $\forall t \in
  [\alpha, \alpha ( +^{n - 1})) . \alpha \leqslant^{n - 1} \eta (n - 1,
  \alpha, t)) + 1$, and since the sequence $\{\eta (n - 1, \alpha, t)) + 1 | t
  \in [\alpha, \alpha ( +^n)) \}$ is confinal in $\alpha (+^{n - 1}$), then by
  (3) of GenThmIH (that is, $\leqslant^{n - 1}$-continuity), $\alpha <^{n - 1}
  \alpha (+^{n - 1}$). This shows (c).

  Finally, $\alpha <_1 \alpha (+^{n - 1}) \Longleftarrow \alpha <^{n - 1}
  \alpha (+^{n - 1})$ clearly holds by (3) of GenThmIH.
\end{proof}

\begin{corollary}
  \label{corollary_kapa_uncount_regul_in_Class(n)}Let $\kappa$ be an
  uncountable regular ordinal ($\kappa \in \tmop{Class} (n - 1)$ by
  proposition \ref{regular_ordinal_in_Class(n-1)}). Then
  \begin{enumeratealpha}
    \item $\kappa <^{n - 1} \kappa (+^{n - 1})$ and therefore $\kappa \in
    \tmop{Class} (n)$.
    
    \item $\kappa \in \bigcap_{s \in [\kappa, \kappa ( +^{n - 1}))} A^{n - 1}
    (s)$.
  \end{enumeratealpha}
\end{corollary}

\begin{proof}
  {\color{orange} Left to the reader.}
\end{proof}

\begin{corollary}
  {\tmdummy}
  
  \begin{enumeratenumeric}
    \item $\tmop{Class} (n) \neq \emptyset$
    
    \item For any $\alpha \in \tmop{Class} (n)$, $\alpha (+^n) < \infty$; that
    is, $\alpha (+^n)$ is an ordinal.
  \end{enumeratenumeric}
\end{corollary}

\begin{proof}
  {\color{orange} Left to the reader.}
\end{proof}

\begin{lemma}
  \label{sequence_csi_j_with_m(csi_j)=t[g(k,q,csi_j)]}Let $k \in [1, n)$, $q
  \in \tmop{Class} (k)$, $t = \eta (k, q, t) \in [q, q ( +^k))$ and $q <_1 t +
  1$. Then there is a sequence $(\xi_j)_{j \in J} \subset \tmop{Class} (k)$
  such that $\xi_j \underset{\text{cof}}{\longhookrightarrow} q$ and such that
  for all $j \in J$, $T (k, q, t) \cap q \subset \xi_j$ and $m (\xi_j) = t [g
  (k, q, \xi_j)]$.
\end{lemma}

\begin{proof}
  Let $k, q$ and $t$ as stated. Then by (6) and (4) of GenThmIH, there exists
  a sequence $(l_i)_{i \in I} \in q \cap \tmop{Class} (k)$, $l_i
  \underset{\text{cof}}{\longhookrightarrow} q$ such that for all $i \in I$,
  \\
  $T (k, q, t) \cap q \subset l_i$ and $m (l_i) \geqslant t [g (k, q, l_i)]$.
  \ \ \ \ (*)

  We have now two cases:
  
  (a). For some subsequence $(l_d)_{d \in D} \subset (l_i)_{i \in I}$ it
  occurs $\forall d \in D.m (l_d) = t [g (k, q, l_d)]$.
  
  Then $(l_d)_{d \in D}$ is the sequence we are looking for.

  (b). For every subsequence $(l_d)_{d \in D} \subset (l_i)_{i \in I}$
  $\exists d \in D.m (l_d) \neq t [g (k, q, l_d)]$.
  
  Choose an arbitrary $e < q$ and let \\
  $l \assign \min \{r \in q \cap \tmop{Class} (k) | T (k, q, t) \cap q \subset
  r > e \wedge m (r) > t [g (k, q, r)]\}$. Observe $l$ exists because of (b)
  and (*). Then \\
  $e < l <_1 t [g (k, q, l)] + 1 = \eta (k, q, t) [g (k, q, l)] + 1 =$\\
  \ \ \ \ \ \ \ \ \ \ \ \ $underset{\text{by (2.2.3) and (2.2.5) of
  GenThmIH}}{=}\eta (k, l, t [g (k, q, l)]) + 1$, which implies, by (6) and
  (4) of GenThmIH, the existence of a sequence $(s_u)_{u \in U}$, $s_u
  \underset{\text{cof}}{\longhookrightarrow} l$ such that for all $u \in U$,
  \\
  $T (k, q, {\color{magenta} \eta (k, q, t)}) \cap q\underset{\text{by (2.2.3) and
  (2.2.4) of GenThmIH}}{=}T (k, l, {\color{magenta} \eta (k, q, t) [g (k, q,
  l)]}) \cap l$\\
  \ \ \ \ \ \ \ \ \ \ \ \ \ \ \ \ \ \ \ \ \ \ \ \ \ \ \ $underset{\text{by (2.2.3)
  and (2.2.5) of GenThmIH}}{=}T (k, l, {\color{magenta} \eta (k, l, t [g (k,
  q, l)])}) \cap l \subset s_u$ \ \ \ \ \ \ \ {\tmstrong{(1*)}} \\
  and\\
  $s_u \leqslant_1 {\color{magenta} \eta (k, l, t [g (k, q, l)])} [g (k, l,
  s_u)]\underset{\text{by (2.2.3) and (2.2.5) of GenThmIH}}{=}$\\
  \ \ \ \ \ \ $\eta (k, q, t) [g (k, q, l)] [g (k, l, s_u)] = \eta (k, q, t)
  [g (k, l, s_u) \circ g (k, q, l)]\underset{\text{by (2.5.3) of GenThmIH}}{=}$\\
  \ \ \ \ \ \ $\eta (k, q, t) [g (k, q, s_u)] = t [g (k, q, s_u)]$. \ \ \ \ \
  \ \ {\tmstrong{(2*)}}

  Now, note that (1*) and (2*) assert $\forall u \in U$.$T (k, q,
  {\color{magenta} \eta (k, q, t)}) \cap q \subset s_u \wedge m (s_u)
  \geqslant t [g (k, q, s_u)]$. Therefore, since $s_u
  \underset{\text{cof}}{\longhookrightarrow} l$, there is some $a \in U$ such
  that $e < s_a < l$, \\
  $T (k, q, t) \cap q \subset s_a$ and $m (s_a) \geqslant t [g (k, q, s_a)]$;
  moreover, by the definition of $l$, \\
  $m (s_a) \ngtr t [g (k, q, s_a)]$ and then $m (s_a) = t [g (k, q, s_a)]$. We
  define $\xi_e \assign s_a$. Then, the sequence $(\xi_e)_{e \in q}$ is the
  sequence we are looking for.
\end{proof}

\subsection{Canonical sequence of an ordinal $e (+^i)$}

{\noindent}{\tmstrong{Reminder:}} For $e \in \mathbbm{E}$, we denote by
$(\omega_k (e))_{k \in \omega}$ to the recursively defined sequence \\
$\omega_0 (e) \assign e + 1$, $\omega_{k + 1} (e) \assign \omega^{\omega_k
(e)}$.

We want now to define, for $e \in \tmop{Class} (i)$, a (canonical) sequence
cofinal in $e (+^i$).

\begin{definition}
  \label{canonical_fundamental_sequence}(Canonical sequence of an ordinal $e
  (+^i))$
  
  For $i \in [1, n)$, $e \in \tmop{Class} (i)$, and $k \in [1, \omega)$ we
  define the set $X_k (i, e)$ and the ordinals $x_k (i, e)$ and $\gamma_k (i,
  e)$ simultaneously by recursion on $([1, n), <)$ as follows:

  Let $i = 1$, $e \in \tmop{Class} (1)$ and $k \in [1, \omega)$. Let it be
  
  $X_k (1, e) \assign \{\omega_k (e)\}$,
  
  $x_k (1, e) \assign \omega_k (e) = \min X_k (1, e)$ and
  
  $\gamma_k (1, e) \assign m (x_k (1, e)) = m (\omega_k (e)) = \pi (\omega_k
  (e)) + d \pi (\omega_k (e)) = \omega_k (e) + d (\omega_k (e)) = \eta
  (\omega_k (e)) =$\\
  \ \ \ \ \ \ \ \ $underset{\text{properties of }\eta}{=}\eta (\eta (\omega_k
  (e))) = \eta (\gamma_k (1, e)) = \eta (1, e, \gamma_k (1, e))$.

  Suppose $i + 1 \in [2, n)$ and that for $i \in [1, n)$, $X_k (i, E)$, $x_k
  (i, E)$ and $\gamma_k (i, E)$ have already been defined for arbitrary $E \in
  \tmop{Class} (i)$ and $k \in [1, \omega)$.

  Let $e \in \tmop{Class} (i + 1)$ and $k \in [1, \omega)$. We define
  
  $X_k (i + 1, e) \assign \{r \in (e, e (+^{i + 1})) \cap \tmop{Class} (i) |
  m (r) = \gamma_k (i, r)\}$,
  
  $x_k (i + 1, e) \assign \min X_k (i + 1, e)$ and
  
  $\gamma_k (i + 1, e) \assign m (x_k (i + 1, e)) \in (x_k (i + 1, e), x_k (i
  + 1, e) ( +^i)) \subset (e, e ( +^{i + 1}))$.

  For $e \in \tmop{Class} (i)$, we call $(\gamma_k (i, e))_{k \in [1,
  \omega)}$ the canonical sequence of $e (+^i)$.
\end{definition}

To assure that our previous definition \ref{canonical_fundamental_sequence} is
correct we need to show that $\min X_k (i, e)$ exists. This is one of the
reasons for our next

\begin{proposition}
  \label{most_important_sequence_in_(e,e(+^i))}$\forall i \in [1, n) \forall e
  \in \tmop{Class} (i)$.
  \begin{enumeratenumeric}
    \item For any $k \in [1, \omega)$, $X_k (i, e) \neq \emptyset$ and
    therefore $\min X_k (i, e)$ exists.
    
    \item $(\gamma_j (i, e))_{j \in [1, \omega)} \subset (e, e ( +^i))$
    
    \item $\forall j \in [1, \omega) . \gamma_j (i, e) = \eta (i, e, \gamma_j
    (i, e))$
    
    \item $(\gamma_j (i, e))_{j \in [1, \omega)}$ and cofinal in $e (+^i)$.
    
    \item If $i = 1$, then $\forall z \in [1, \omega) .T (i, e, \gamma_z (i,
    e)) = \{\lambda (1, \gamma_z (1, e)) = e\}$.\\
    Case $i \geqslant 2$. Then for any $z \in [1, \omega)$,
    
    - $m (x_z (i, e)) = m (x_z (i - 1, {\color{magenta} x_z (i, e)})) = \ldots
    = m (x_z (2, x_z (3, \ldots x_z (i - 1, {\color{magenta} x_z (i, e)})
    \ldots)))$;
    
    - $T (i, e, \gamma_z (i, e)) = \{o_1 > o_2 > \ldots > o_{i - 1} > o_i =
    e\}$, where
    
    \ $o_1 \assign \lambda (1, \gamma_z (i, e))$, \\
    \ \ \ \ \ $o_2 \assign \lambda (2, \gamma_z (i, e)), \ldots$,\\
    \ \ \ \ \ $o_{i - 1} \assign \lambda (i - 1, \gamma_z (i, e))$,\\
    \ \ \ \ \ $o_i \assign \lambda (i, \gamma_z (i, e)) = e$;
    
    - Moreover,
    
    \ $x_z (i, e) = o_{i - 1}$,$x_z (i - 1, {\color{magenta} x_z (i, e)}) =
    o_{i - 2}, \ldots$,$x_z (2, x_z (3, \ldots x_z (i - 1, {\color{magenta}
    x_z (i, e)}) \ldots)) = o_1$
    
    \ and
    
    \ $o_2 = \lambda (2, o_1), o_3 = \lambda (3, o_2), \ldots, o_i = \lambda
    (i, o_{i - 1}), o_{i + 1} = \lambda (i + 1, o_i)$.
    
    \item $\forall j \in [1, \omega) . \forall a \in ( {\color{magenta} T (i,
    e, \gamma_j (i, e))} \backslash \{e\}) .m (a) = \gamma_j (i, e)$
    
    \item $\forall \alpha \in \tmop{Class} (i) . \forall j \in [1, \omega) .
    \emptyset = T (i, e, \gamma_j (i, e)) \cap e \subset \alpha \wedge
    \gamma_j (i, e) [g (i, e, \alpha)] = \gamma_j (i, \alpha)$
  \end{enumeratenumeric}
\end{proposition}

\begin{proof}
  We prove simultaneously 1, 2, 3, 4, 5, 6 and 7 by induction on $[1, n)$.

  {\tmstrong{Case $i = 1$ and $e \in \tmop{Class} (1)$.}}

  It follows immediately from definition \ref{canonical_fundamental_sequence}
  (and the equalities explicitly given there) that 1, 2 and 3 hold. Moreover,
  it is also clear that 4. holds.
  
  Now, let $j \in [1, \omega)$ be arbitrary. Then, by the definition (see
  statement of theorem \ref{most_most_general_theorem}), \\
  $T (1, e, \gamma_j (1, e)) = \bigcup_{E \in \tmop{Ep} (\gamma_j (1, e))} T
  (1, e, E) = \tmop{Ep} (\gamma_j (1, e)) = \{e = \lambda (1, \gamma_j (1,
  e))\}$. So 5. holds. Moreover, by the equality $T (1, e, \gamma_j (1, e)) =
  \{e\}$ it is clear that 6. holds too.
  
  Finally, let $\alpha \in \tmop{Class} (1)$ and $j \in [1, \omega)$ be
  arbitrary. Then by 5. $\emptyset = T (i, e, \gamma_j (i, e)) \cap e \subset
  \alpha$. Moreover, by definition of $\gamma_j (1, e)$ and the usual
  properties of the substitution $x \longmapsto x [e \assign \alpha]$, we have
  $\gamma_j (1, e) [g (1, e, \alpha)] = \gamma_j (1, e) [e \assign \alpha] =
  \gamma_j (1, \alpha)$, that is, 7. holds.

  {\tmstrong{Let $i + 1 \in [2, n)$}} and suppose the claim holds for $i$. \ \
  \ \ \ \ \ {\tmstrong{(IH)}}

  Let $e \in \tmop{Class} (i + 1)$.

  {\underline{To show that 1. holds.}}
  
  Let $k \in [1, \omega)$.
  
  Since $e (+^{i + 1}) \in \tmop{Class} (i + 1) \subset \tmop{Class} (i)$,
  then by our (IH), \\
  $\eta (i, e ( +^{i + 1}), \gamma_k (i, e ( +^{i + 1}))) = \gamma_k (i, e (
  +^{i + 1})) \in (e ( +^{i + 1}), e (+^{i + 1}) ( +^i))$; from this, the fact
  that $e (+^{i + 1}) <_1 e (+^{i + 1}) ( +^i$) and $<_1$-connectedness
  follows \\
  $e (+^{i + 1}) <_1 \eta (i, e ( +^{i + 1}), \gamma_k (i, e ( +^{i + 1}))) +
  1$. Thus, by lemma \ref{sequence_csi_j_with_m(csi_j)=t[g(k,q,csi_j)]}, there
  is a sequence \\
  $(\xi_j)_{j \in J} \subset \tmop{Class} (i)$ such that $\xi_j
  \underset{\text{cof}}{\longhookrightarrow} e (+^{i + 1}$) and such that for
  all $j \in J$, \\
  $T (i, e ( +^{i + 1}), \gamma_k (i, e ( +^{i + 1}))) \cap e (+^{i + 1})
  \subset \xi_j$ and \\
  $m (\xi_j) = \gamma_k (i, e ( +^{i + 1})) [g (i, e (+^{i + 1}),
  \xi_j)]\underset{\text{by 7. of our (IH)}}{=}\gamma_k (i, \xi_j)$.
  
  From the previous follows that \\
  $X_k (i + 1, e) = \{r \in (e, e (+^{i + 1})) \cap \tmop{Class} (i) | m (r) =
  \gamma_k (i, r)\} \neq \emptyset$. Hence 1. holds.

  {\underline{2. holds.}}

  This is clear from the definition of $(\gamma_j (i + 1, e))_{j \in [1,
  \omega)}$ (the fact that $X_k (i + 1, e) \neq \emptyset$ implies that
  $(\gamma_j (i + 1, e))_{j \in [1, \omega)}$ is well defined).

  {\underline{To show that $[\gamma_k (i + 1, e))_{k \in [1, \omega)}$
  satisfies 3.}}

  Let $k \in [1, \omega)$.
  
  Since $x_k (i + 1, e) \in (e, e ( +^{i + 1})) \cap \tmop{Class} (i)$, then
  $x_k (i + 1, e) \geqslant e (+^i)$ and so \\
  $m (x_k (i + 1, e)) \geqslant e (+^i) ( +^{i - 1}) \ldots ( +^1) 2$. \ \ \ \
  \ \ \ (1*)
  
  On the other hand, for any $t \in (e, x_k (i + 1, e))$ proposition
  \ref{alpha<less>_1_beta_in_Class(k)_then_alpha_in_Class(k+1)} implies \\
  $m (t) < x_k (i + 1, e) \leqslant {\color{magenta} m (x_k (i + 1, e))}$.
  Moreover, notice for any $t \in [x_k (i + 1, e), {\color{magenta} m (x_k (i
  + 1, e))}]$, $m (t) \ngtr {\color{magenta} m (x_k (i + 1, e))}$: Assume the
  opposite. Then the inequalities \\
  $x_k (i + 1, e) \leqslant t \leqslant {\color{magenta} m (x_k (i + 1, e))} <
  {\color{magenta} m (x_k (i + 1, e))} + 1 \leqslant m (t)$ imply by
  $\leqslant_1$-connectedness that $x_k (i + 1, e) \leqslant_1 t <_1
  {\color{magenta} m (x_k (i + 1, e))} + 1$ and then, by
  $\leqslant_1$-transitivity, \\
  $x_k (i + 1, e) <_1 {\color{magenta} m (x_k (i + 1, e))} + 1$.
  Contradiction. Hence, from all this we conclude \\
  $\forall t \in (e, {\color{magenta} m (x_k (i + 1, e))}] .m (t) \leqslant
  {\color{magenta} m (x_k (i + 1, e))}$. \ \ \ \ \ \ \ (2*)
  
  Finally, \\
  $\eta (i + 1, e, \gamma_k (i + 1, e)) = \eta (i + 1, e, m (x_k (i + 1,
  e)))\underset{\text{by (1*)}}{=}$\\
  \ \ \ \ \ \ \ \ \ \ \ \ \ \ \ \ \ \ \ \ \ \ \ \ \ \ \ $= \max \{m (\beta) |
  \beta \in (e, m (x_k (i + 1, e))]\}\underset{\text{by (2*)}}{=}$
  
  \ \ \ \ \ \ \ \ \ \ \ \ \ \ \ \ \ \ \ \ \ \ $= m (x_k (i + 1, e)) =
  \gamma_k (i + 1, e)$.\\
  Thus 3. holds.

  {\underline{To show 4., that is, $(\gamma_k (i + 1, e))_{k \in [1, \omega)}$
  is cofinal in $e (+^{i + 1})$.}}

  First note that since $e (+^{i + 1}) {\color{magenta} (+^{i - 1})} \ldots (
  +^2) ( +^1) 2 + 1 \in (e (+^{i + 1}), e (+^{i + 1}) ( +^i))$, then\\
  $e (+^{i + 1}) <_1 e (+^{i + 1}) {\color{magenta} (+^{i - 1})} \ldots ( +^2)
  ( +^1) 2 + 1 =$\\
  $\eta (i, e ( +^{i + 1}), e (+^{i + 1}) {\color{magenta} (+^{i - 1})} \ldots
  ( +^2) ( +^1) 2) + 1$ by $<_1$-connectedness; then, by (6) and (4) of
  GenThmIH, there exist a sequence of elements in $\tmop{Class} (i)$ that is
  cofinal in $e (+^{i + 1}$). So, to show that $(\gamma_k (i + 1, e))_{k \in
  [1, \omega)}$ is cofinal in $e (+^{i + 1}$) it is enough \\
  to show $\forall \sigma \in (e, e ( +^{i + 1})) \cap \tmop{Class} (i) .
  \exists s \in [1, \omega) . \gamma_s (i + 1, e) > \sigma$. \ \ \ \ \ \ \
  {\tmstrong{(b1)}}

  To show (b1).
  
  Let $\sigma \in \tmop{Class} (i) \cap (e, e ( +^{i + 1}))$. Then by (1.3),
  (1.3.1), (1.3.5), (1.3.6) and (1.3.4) of GenThmIH $f (i + 1, e) (\sigma) =
  \{\sigma = \sigma_1 > \ldots > \sigma_q \}$ for some $q \in [1, \omega)$, \
  \ \ \ \ \ \ {\tmstrong{(c0)}}\\
  where \\
  $\sigma_q = \min \{d \in (e, \sigma_q] \cap \tmop{Class} (i) | m (d) [g (i,
  d, \sigma_q)] \geqslant m (\sigma_g)\}$, \ \ \ \ \ \ \ {\tmstrong{(c1)}}\\
  $\forall l \in [1, q - 1] . \sigma_l = \min \{d \in (\sigma_{l + 1},
  \sigma_l] \cap \tmop{Class} (i) | m (d) [g (i, d, \sigma_l)] \geqslant m
  (\sigma_l)\}$ \ \ \ \ \ \ \ {\tmstrong{(c2)}}\\
  and \\
  $m (\sigma) = m (\sigma_1) \leqslant m (\sigma_2) [g (i, \sigma_2, \sigma)]
  \leqslant m (\sigma_3) [g (i, \sigma_3, \sigma)] \leqslant \ldots \leqslant
  m (\sigma_q) [g (i, \sigma_q, \sigma)]$. \ \ \ \ \ \ \ {\tmstrong{(c3)}}

  On the other hand, by (IH) $(\gamma_j (i, \sigma_q))_{j \in [1, \omega)}$ is
  cofinal in $\sigma_q (+^i$), so there exists $z \in [1, \omega)$ such that
  $\gamma_z (i, \sigma_q) \in (m (\sigma_q), \sigma_q ( +^i))$. \ \ \ \ \ \ \
  {\tmstrong{(c4)}}\\
  But by (c1), \\
  $\forall d \in (e, \sigma_q] \cap \tmop{Class} (i)$.$m (d) [g (i, d,
  \sigma_q)] \leqslant m (\sigma_q) < \gamma_z (i, \sigma_q)$, particularly,
  \\
  $\forall d \in (e, \sigma_q] \cap \tmop{Class} (i)$.$m (d) [g (i, d,
  \sigma_q)] < \gamma_z (i, \sigma_q)$. From this and using (2.3.2) of
  GenThmIH and (7) of our (IH), we get \\
  $\forall d \in (e, \sigma_q] \cap \tmop{Class} (i)$.\\
  $m (d) = m (d) [g (i, d, \sigma_q)] [g (i, \sigma_q, d)] < \gamma_z (i,
  \sigma_q) [g (i, \sigma_q, d)]\underset{\text{by (7) of our (IH)}}{=}\gamma_z
  (i, d)$ \ \ \ \ \ \ \ {\tmstrong{(*)}}

  Now let $l \in [1, q - 1]$ and $d \in (\sigma_{l + 1}, \sigma_l] \cap
  \tmop{Class} (i)$. By (c2), $m (d) [g (i, d, \sigma_l)] \leqslant m
  (\sigma_l)$; this inequality, (c0) and (2.5.3) and (2.3.1) of GenThmIH
  imply,\\
  $m (d) [g (i, d, \sigma)] = m (d) [g (i, d, \sigma_l)] [g (i, \sigma_l,
  \sigma)] \leqslant m (\sigma_l) [g (i, \sigma_l, \sigma)] \leqslant$\\
  \ \ \ \ \ \ \ \ \ \ \ \ \ \ \ \ \ \ $underset{\text{by (c3)}}{\leqslant}m
  (\sigma_q) [g (i, \sigma_q, \sigma)]\underset{\text{using (c4)}}{<}\gamma_z (i,
  \sigma_q) [g (i, \sigma_q, \sigma)]\underset{\text{(7) of our (IH)}}{=}\gamma_z
  (i, \sigma)$. From this, by (2.3.2) and (2.5.3) of GenThmIH, \\
  $m (d) = m (d) [g (i, d, \sigma)] [g (i, \sigma, d)] < \gamma_z (i, \sigma)
  [g (i, \sigma, d)]\underset{\text{(7) of our (IH)}}{=}\gamma_z (i, d)$. The
  previous shows $\forall l \in [1, q - 1] \forall d \in (\sigma_{l + 1},
  \sigma_l] \cap \tmop{Class} (i) .m (d) < \gamma_z (i, d)$. \ \ \ \ \ \ \
  {\tmstrong{(**)}}

  From (*) and (**) follows $\forall d \in (e, \sigma] \cap \tmop{Class} (i)
  .m (d) < \gamma_z (i, d)$, and therefore\\
  $\forall d \in (e, \sigma] \cap \tmop{Class} (i) .d < \min X_z (i + 1, e) =
  x_z (i + 1, e) \leqslant m (x_z (i + 1, e)) = \gamma_z (i + 1, e)$.
  
  This shows (b1). Hence 4. holds.

  {\underline{To show that $(\gamma_k (i + 1, e))_{k \in [1, \omega)}$
  satisfies 5.}}

  First note that for arbitrary $k, j \in [1, \omega)$ and $c \in \tmop{Class}
  (j + 1)$\\
  $x_k (j + 1, c) = \min \{r \in (c, c (+^{j + 1})) \cap \tmop{Class} (j) | m
  (r) = \gamma_k (j, r)\}$. \ \ \ \ \ \ \ {\tmstrong{(J0)}}\\
  So\\
  $m (x_k (j + 1, c)) = \gamma_k (j, x_z (j + 1, c)) = m (x_k (j, x_z (j + 1,
  c)))$; \ \ \ \ \ \ \ {\tmstrong{(J1)}}\\
  $x_k (j + 1, c) \in \tmop{Class} (j) \backslash \tmop{Class} (j + 1)$; \ \ \
  \ \ \ \ {\tmstrong{(J2)}}\\
  $\gamma_k (j + 1, c) = m (x_k (j + 1, c)) \in (x_k (j + 1, c), x_k (j + 1,
  c) ( +^j))$; \ \ \ \ \ \ \ {\tmstrong{(J3)}}\\
  $\lambda (j, m (x_k (j + 1, c)) = \lambda (j, \gamma_k (j + 1, c)) = x_k (j
  + 1, c)$. \ \ \ \ \ \ \ {\tmstrong{(J4)}}

  Let $z \in [1, \omega)$.

  We show now
  
  $\gamma_z (i + 1, e) = m (x_z (i + 1, e)) = m (x_z (i, {\color{magenta} x_z
  (i + 1, e)})) = \ldots =$\\
  \ \ \ \ \ \ \ \ \ \ \ \ \ \ \ \ \ \ \ \ \ \ \ \ \ \ \ \ \ \ \ \ \ \ \ \ \ \
  \ $= m (x_z (2, x_z (3, \ldots x_z (i, {\color{magenta} x_z (i + 1, e)})
  \ldots)))$ \ \ \ \ \ \ \ {\tmstrong{(J5)}}\\
  This is easy:\\
  $\gamma_z (i + 1, e) = m (x_z (i + 1, e))\underset{\text{by (J1)}}{=}m (x_z (i,
  x_z (i + 1, e)))\underset{\text{by (J1)}}{=}$\\
  \ \ \ \ \ \ \ \ \ \ \ \ \ $= m (x_z (i - 1, x_z (i, x_z (i + 1,
  e))))\underset{\text{by (J1)}}{=}\ldots\underset{\text{by (J1)}}{=}$
  
  \ \ \ \ \ \ \ \ \ $= m (x_z (2, x_z (3, \ldots x_z (i, {\color{magenta} x_z
  (i + 1, e)}) \ldots)))$.
  
  This shows (J5).

  Let's abbreviate\\
  $o_1 \assign \lambda (1, \gamma_z (i + 1, e))$,\\
  $o_2 \assign \lambda (2, \gamma_z (i + 1, e)), \ldots$,\\
  $o_i \assign \lambda (i, \gamma_z (i + 1, e))$,\\
  $o_{i + 1} \assign \lambda (i + 1, \gamma_z (i + 1, e))$. \ \ \ \ \ \ \
  {\tmstrong{(d1)}}

  To show $o_{i + 1} = e$,
  
  $x_z (i + 1, e) = o_i$,$x_z (i, {\color{magenta} x_z (i + 1, e)}) = o_{i -
  1}, \ldots$,$x_z (2, x_z (3, \ldots x_z (i, {\color{magenta} x_z (i + 1,
  e)}) \ldots)) = o_1$ and
  
  $o_2 = \lambda (2, o_1), o_3 = \lambda (3, o_2), \ldots, o_i = \lambda (i,
  o_{i - 1}), o_{i + 1} = \lambda (i + 1, o_i)$. \ \ \ \ \ \ \
  {\tmstrong{(J6)}}

  First let's see $o_{i + 1} = e$. \ \ \ \ \ \ \ {\tmstrong{(J6.1)}}
  
  Note $\gamma_z (i + 1, e)\underset{\text{By (J3)}}{\in}(x_z (i + 1, e), x_z (i
  + 1, e) ( +^i))\underset{\text{By (J0)}}{\subset}(e, e ( +^{i + 1}))$. Then,
  since \\
  $(e, e ( +^{i + 1})) \cap \tmop{Class} (i) = \emptyset$, we get $o_{i + 1} =
  \lambda (i + 1, \gamma_z (i + 1, e)) = e$. So (J6.1) holds.

  Now let's show $x_z (i + 1, e) = o_i, \ldots$,$x_z (2, x_z (3, \ldots x_z
  (i, {\color{magenta} x_z (i + 1, e)}) \ldots)) = o_1$. \ \ \ \ \ \ \
  {\tmstrong{(J6.2)}}
  
  This is also easy:\\
  $x_z (i + 1, e)\underset{\text{by (J4)}}{=}\lambda (i, \gamma_z (i + 1, e)) =
  o_i$,\\
  $x_z (i, x_z (i + 1, e))\underset{\text{by (J4)}}{=}\lambda (i - 1,
  {\color{magenta} m (x_z (i, x_z (i + 1, e)))})\underset{\text{by 
  (J5)}}{=}\lambda (i - 1, \gamma_z (i + 1, e)) = o_{i - 1}$,\\
  $\ldots$\\
  $x_z (2, x_z (3, \ldots x_z (i, {\color{magenta} x_z (i + 1, e)})
  \ldots))\underset{\text{by (J4)}}{=}\lambda (1, m (x_z (2, x_z (3, \ldots x_z
  (i, {\color{magenta} x_z (i + 1, e)}) \ldots))))\underset{\text{by (J5)}}{=}$\\
  \ \ \ \ \ \ \ \ \ \ \ \ \ \ \ \ \ \ \ \ \ \ \ \ \ \ \ \ \ \ \ \ \ \ \ \ \ \
  \ \ \ \ \ \ $= \lambda (1, \gamma_z (i + 1, e)) = o_1$. \\
  So (J6.2) holds.

  Let's see that $o_2 = \lambda (2, o_1), o_3 = \lambda (3, o_2), \ldots, o_i
  = \lambda (i, o_{i - 1}), o_{i + 1} = \lambda (i + 1, o_i)$. \ \ \ \ \ \ \
  {\tmstrong{(J6.3)}}
  
  Note for any $k \in [1, i]$
  
  $o_k \underset{\text{by (J 6.2) and (J 6.1)}}{=} x_z (k + 1,
  o_{k + 1}) \underset{\text{by (J 0)}}{\in} (o_{k + 1}, o_{k + 1} (
  +^{k + 1})) \cap \tmop{Class} (k)$, so $\lambda (k + 1, o_k) = o_{k + 1}$.\\
  So (J6.3) holds.

  Hence (J6) holds because of the proofs of (J6.1), (J6.2), (J6.3).

  To show $T (i + 1, e, \gamma_z (i + 1, e)) = \{o_1 > o_2 > \ldots > o_{i -
  1} > o_{i + 1} \}$ \ \ \ \ \ \ \ {\tmstrong{(J7)}}

  Since $\gamma_z (i + 1, e)\underset{\text{by (J5) and (J6)}}{=}m (o_1)$, then
  \\
  $T (i + 1, e, \gamma_z (i + 1, e)) = T (i + 1, e, m (o_1)) = \bigcup_{d \in
  \tmop{Ep} (m (o_1))} T (i + 1, e, d)\underset{\tmop{Ep} (m (o_1)) = T (1,
  o_1, m (o_1))}{=}$\\
  $\bigcup_{d \in T (1, o_1, m (o_1))} T (i + 1, e, d)\underset{\text{by our 
  (IH)}}{=}\bigcup_{d \in \{o_1 \}} T (i + 1, e, d) = T (i + 1, e, o_1) =
  \bigcup_{k \in \omega} O (k, o_1)$,

  where by definition \\
  $E_1 = \lambda (1, m (o_1)) = o_1,E_2 = \lambda (2, E_1)\underset{\text{by 
  (d1)}}{=}o_2, \ldots,E_{i + 1} = \lambda (i + 1, E_i)\underset{\text{by 
  (d1)}}{=}o_{i + 1}$ and

  $O (0, o_1) \assign\underset{\delta \in W (0, k, o_1), k = 1, \ldots, i}{\bigcup}
  f (k + 1, \lambda (k + 1, \delta)) (\delta) \cup
  \tmop{Ep} (m (\delta)) \cup \{\lambda (k + 1, \delta)\}$;\\
  $W (0, k, o_1) \assign (e, e ( +^{i + 1})) \cap \{E_1 > E_2 \geqslant E_3
  \geqslant \ldots \geqslant E_{i + 1} = e\} \cap (\tmop{Class} (k) \backslash
  \tmop{Class} (k + 1))$;

  $O (l + 1, o_1) \assign\underset{\delta \in W (l, k, o_1), k = 1, \ldots, i}{\bigcup}
  f (k + 1, \lambda (k + 1, \delta)) (\delta) \cup
  \tmop{Ep} (m (\delta)) \cup \{\lambda (k + 1, \delta)\}$;

  $W (l, k, o_1) \assign (e, e ( +^{i + 1})) \cap O (l, o_1) \cap
  (\tmop{Class} (k) \backslash \tmop{Class} (k + 1))$.

  But note for any $k \in [1, i]$,\\
  $W (0, k, o_1) = \{o_k \}$,\\
  $\lambda (k + 1, o_k) = o_{k + 1}$,\\
  $f (k + 1, \lambda (k + 1, o_k)) (o_k) = f (k + 1, o_{k + 1}) (o_k) = \{o_k
  \}$,\\
  $\tmop{Ep} (m (o_k))\underset{\text{by (J5) and (J6)}}{=}\tmop{Ep} (m (o_1)) =
  \{o_1 \}$. \\
  Therefore $f (k + 1, \lambda (k + 1, o_k)) (o_k) \cup \tmop{Ep} (m (o_k))
  \cup \{\lambda (k + 1, o_k)\} = \{o_k, o_1, o_{k + 1} \}$. This way $O (0,
  o_1) = \{o_1 > o_2 \geqslant o_3 \geqslant \ldots \geqslant o_{i + 1} =
  e\}$, and moreover, exactly because of the same reasoning, $\forall l \in
  \omega .O (l + 1, o_1) = \{o_1 > o_2 \geqslant o_3 \geqslant \ldots
  \geqslant o_{i + 1} = e\}$. Thus, we conclude \\
  $T (i + 1, e, \gamma_z (i + 1, e)) = \{o_1 > o_2 \geqslant o_3 \geqslant
  \ldots \geqslant o_{i + 1} = e\}$. Finally, we just make the reader aware
  that actually $o_1 > o_2 > o_3 > \ldots > o_{i + 1}$ holds because of (J6)
  and (J2). So we have shown (J7). This concludes the proof of 5.

  {\underline{To show that $(\gamma_k (i + 1, e))_{k \in [1, \omega)}$
  satisfies 6.}}

  From 5. we get $T (i, e, \gamma_z (i, e)) \backslash \{e\} = \{o_1 > o_2 >
  \ldots > o_{i - 1} \}$ with $m (o_{i - 1}) = \ldots = m (o_1)$.

  {\underline{To show $(\gamma_k (i + 1, e))_{k \in [1, \omega)}$ satisfies
  7.}}

  Let $\alpha \in \tmop{Class} (i + 1)$ and $z \in [1, \omega)$.
  
  Let $o_1, \ldots, o_{i + 1}$ as in 5. (that is, for $k \in [1, i + 1]$, $o_k
  \assign \lambda (k, {\color{magenta} \gamma_z (i + 1, e)})$). By 5., we know
  that $T (i + 1, e, \gamma_z (i + 1, e)) = \{o_1 > o_2 > \ldots > o_{i - 1} >
  o_{i + 1} = e\}$. So $T (i + 1, e, \gamma_z (i + 1, e)) \cap e = \emptyset$.
  Now, for any $k \in [1, i + 1]$, $T (i + 1, e, o_k)\underset{\text{definition of }
  T (i + 1, e, o_k)}{\subset}T (i + 1, e, \gamma_z (i + 1, e))$, which
  means \\
  $\forall k \in [1, i + 1] . \emptyset = T (i + 1, e, o_k) \cap e \subset
  \alpha$. The latter expression implies, by (2.2.3) of GenThmIH that $\forall
  k \in [1, i + 1] . \tmop{Ep} (o_k) \subset \tmop{Dom} [g (i + 1, e,
  \alpha)]$. So for $k \in [1, i + 1]$, let $\tmmathbf{u_k \assign o_k [g (i +
  1, e, \alpha)]}$.

  We will need the following observations (K1), (K2), (K3) and (W):

  Since by 5. we know $\forall k \in [1, i] .o_{k + 1} = \lambda (k + 1,
  o_k)$, then by (2.4.6) of GenThmIH, this implies $\forall k \in [1, i] .u_{k
  + 1} = o_{k + 1} [g (i + 1, e, \alpha)] = \lambda (k + 1, o_k [g (i + 1, e,
  \alpha)]) = \lambda (k + 1, u_k)$. \ \ \ \ \ \ \ {\tmstrong{(K1)}}

  Note $o_1\underset{\text{by 5.}}{=}x_z (2, o_2) \in X_z (2, o_2) \assign \{r
  \in (o_2, o_2 (+^2)) \cap \tmop{Class} (1) | m (r) = \gamma_z (1, r)\}$.
  This implies $m (o_1) = \gamma_z (1, o_1)$. \ \ \ \ \ \ \ {\tmstrong{(K2)}}

  Moreover, observe that\\
  $o_i <_1 o_{i - 1} <_1 \ldots <_1 o_1 <_1 \gamma_z (i + 1, e) = m (o_1)
  \underset{\text{by (K 2)}}{=} \gamma_z (1, o_1)
  \underset{\text{by definition}}{=} m (\omega_z (o_1))$ and \\
  $\forall j \in [1, i] .o_j \nleqslant_1 m (\omega_z (o_1)) + 1$ imply, by
  (2.4.3) of GenThmIH, that\\
  $u_i <_1 \ldots <_1 u_1 <_1 m (\omega_z (o_1)) [g (i + 1, e, \alpha)]
  \underset{\text{by (2.4.4) of GenThmIH}}{=} m (
  {\color{magenta} (\omega_z (o_1)) [g (i + 1, e, \alpha)]}) =$ \ \ \ \ \ \ \
  \ \ \ \ \ \ \ \ \ \ \ \ \ \ \ \ \ \ \ \ \ \ \ \ \ \ \ \ \ \ \ \ \ \ \ \ \ \
  \ \ \ \ \ \ \ \ \ \ \ \ \ \ \ \ \ \ \ \ \ \ \ \ \ \ \ \ \ $= m (\omega_z
  (u_1))$\\
  and \\
  $\forall j \in [1, i] .u_j \nleqslant_1 (m (\omega_z (o_1)) + 1) [g (i + 1,
  e, \alpha)] \underset{\text{by (2.4.4) of GenThmIH}}{=}$\\
  \ \ \ \ \ \ \ \ \ \ \ \ \ \ \ \ \ \ \ \ $m ( {\color{magenta} (\omega_z
  (o_1)) [g (i + 1, e, \alpha)]}) + 1 = m (\omega_z (u_1)) + 1$.\\
  From this follows $\forall j \in [1, i] .m (u_j) = m (\omega_z (u_1))
  \underset{\text{by definition}}{=} \gamma_z (1, u_1)$. \ \ \ \
  \ \ \ {\tmstrong{(K3)}}

  Now we show that $\forall j \in [1, i] .u_j = x_z (j + 1, u_{j + 1})$ \ \ \
  \ \ \ \ {\tmstrong{(W)}}

  We prove (W) by a (side)induction on $([1, i], <)$.

  Let $j \in [1, i]$.
  
  Suppose $\forall l \in j \cap [1, i] .u_l = x_z (l + 1, u_{l + 1})$. \ \ \ \
  \ \ \ {\tmstrong{(WIH)}}

  Note $m (u_j)\underset{\text{by (K3)}}{=}\gamma_z (1, u_1)\underset{\text{if }j
  \geqslant 2}{=}m (u_{j - 1})\underset{\text{by (WIH)}}{=}m (x_z (j, u_j)) =
  \gamma_z (j, u_j)$. This shows that, in any case, $m (u_j) = \gamma_z (j,
  u_j)$ \ \ \ \ \ {\tmstrong{(M1)}}.\\
  This way, \\
  $u_j \underset{\text{by (M1)}}{\in}  \{r \in (\lambda (j + 1, u_j),
  \lambda (j + 1, u_j) (+^{j + 1})) \cap \tmop{Class} (j) |m (r) = \gamma_z
  (j, r)\} =$\\
  \ \ \ \ \ \ \ \ \ \ \ $= X_z (j + 1, \lambda (j + 1, u_j))\underset{\text{by
  (K1)}}{=}X_z (j + 1, u_{j + 1})$. Moreover, \ since \\
  $f (j + 1, o_{j + 1}) (o_j) = \{o_j \}\underset{\text{by (2.4.5) of
  GenThmIH}}{\Longleftrightarrow}f (j + 1, u_{j + 1}) (u_j) = \{u_j \}$,
  then by (1.3.5) of GenThmIH, \\
  $u_j = \min \{s \in (u_{j + 1}, u_{j + 1} (+^{j + 1})) \cap \tmop{Class} (j)
  | s \leqslant u_j \wedge m (s) [g (j, s, u_j)] = m (u_j)\} =$\\
  \ \ $= \min \{s \in (u_{j + 1}, u_{j + 1} (+^{j + 1})) \cap \tmop{Class}
  (j) | s \leqslant u_j \wedge m (s) [g (j, s, u_j)] \underset{\text{by (M1)}}{=} \gamma_z (j, u_j)\} =$\\
  \ \ $=$, since by (IH) 7., applied to $j \leqslant i$, $u_j, s \in
  \tmop{Class} (j)$, we get $T (j, u_j, \gamma_z (j, u_j)) \cap u_j =
  \emptyset \subset s$,\\
  \ \ $= \min \{s \in (u_{j + 1}, u_{j + 1} (+^{j + 1})) \cap \tmop{Class}
  (j) | s \leqslant u_j \wedge m (s) = \gamma_z (j, u_j) [g (j, u_j, s)]\}$\\
  \ \ $=$, by (IH) 7. applied to $j \leqslant i$, $u_j, s \in \tmop{Class}
  (j)$ and $z \in [1, \omega)$,\\
  \ \ $= \min \{s \in (u_{j + 1}, u_{j + 1} (+^{j + 1})) \cap \tmop{Class}
  (j) | s \leqslant u_j \wedge m (s) = \gamma_z (j, s)\} =$\\
  \ \ $= \min \{s \in (u_{j + 1}, u_{j + 1} (+^{j + 1})) \cap \tmop{Class}
  (j) | m (s) = \gamma_z (j, s)\} =$\\
  \ \ $= \min X_z (j + 1, u_{j + 1}) = x_z (j + 1, u_{j + 1})$.\\
  This shows that (W) holds.

  Finally, \\
  $\gamma_z (i + 1, e) [g (i + 1, e, \alpha)] = m (x_z (i + 1, e)) [g (i + 1,
  e, \alpha)] \underset{\text{\tmop{by} 5.}}{=} m (o_i) [g (i + 1, e, \alpha)]
  =$\\
  \ \ \ \ \ \ \ \ \ \ \ \ \ \ \ \ \ \ \ \ \ \ \ \ \ \ \ \ \ \ $= m (o_i [g (i
  + 1, e, \alpha)]) = m (u_i) \underset{\text{by (W)}}{=} m (x_z (i +
  1, u_{i + 1})) =$\\
  \ \ \ \ \ \ \ \ \ \ \ \ \ \ \ \ \ \ \ \ \ \ \ \ \ \ \ \ \ \ $= m (x_z (i +
  1, \alpha)) = \gamma_z (i + 1, \alpha)$.\\
  Since this last equality was done for arbitrary $\alpha \in \tmop{Class} (i
  + 1)$ and $z \in [1. \omega)$, then 7. holds.
\end{proof}

\begin{remark}
  \label{Remark_Cannonical_sequences}For $i \in [1, n)$ and $e \in
  \tmop{Class} (i)$, it is not hard to see that the sequences \\
  $(x_k (i, e))_{k \in [1, \omega)}$ and $(\gamma_k (i, e))_{k \in [1,
  \omega)}$ are strictly increasing.
  
  Moreover, for any $k \in [1, \omega)$, $\eta (i, e, x_k (i, e)) = m (x_k
  (i, e))$. This equality holds because \\
  $x_k (i, e) \leqslant m (x_k (i, e))$, implies \\
  $m (x_k (i, e)) \leqslant \eta (i, e, x_k (i, e)) \leqslant \eta (i, e, m
  (x_k (i, e))\underset{\text{by 3. of previous proposition }
  \ref{most_important_sequence_in_(e,e(+^i))}}{=}m (x_k (i, e))$.
\end{remark}

\subsection{$\tmop{Class} (n)$ is $\kappa$-club}

\begin{proposition}
  \label{alpha_in_Class(n)_iff_alpha_in_certain_intersection_A^n-1(s)}Let
  $\kappa$ be an uncountable regular ordinal and $r \in \tmop{Class} (n - 1)
  \cap \kappa$ arbitrary. Let $M^{n - 1} (r, \kappa) \assign \{q \in [\kappa,
  \kappa (+^{n - 1})) | T (n - 1, \kappa, q) \cap \kappa \subset r\}$. Then
  $\bigcap_{s \in M^{n - 1} (r, \kappa)} A^{n - 1} (s) \subset \tmop{Class}
  (n)$.
\end{proposition}

\begin{proof}
  Let $\alpha \in \bigcap_{s \in M^{n - 1} (r, \kappa)} A^{n - 1} (s)$
  
  Consider $(\gamma_j (n - 1, \kappa))_{j \in [1, \omega)}$, the canonical
  sequence of $\kappa (+^{n - 1}) \in \tmop{Class} (n - 1)$. Then \\
  $\forall j \in [1, \omega) . \gamma_j (n - 1, \kappa) \in M^{n - 1} (r,
  \kappa)$. Therefore for any $j \in [1, \omega)$, \\
  $\alpha \in A^{n - 1} (\gamma_j (n - 1, \kappa))\underset{\text{theorem }
  \ref{G^n-1(t)=A^n-1(t)_Gen_Hrchy_thm}}{=}G^{n - 1} (\gamma_j (n - 1,
  \kappa))$. \\
  \ \ \ \ \ \ \ \ \ \ \ \ \ \ \ \ \ \ \ \ \ \ \ \ \ \ \ \ \ \ \ \ \ \
  $=\{\beta \in \tmop{Class} (n - 1) | T (n - 1, \kappa, \gamma_j (n - 1,
  \kappa)) \cap \kappa \subset \beta \leqslant \kappa \wedge$ \\
  \ \ \ \ \ \ \ \ \ \ \ \ \ \ \ \ \ \ \ \ \ \ \ \ \ \ \ \ \ \ \ \ \ \ \ \ \ \
  \ $\beta \leqslant^{n - 1} \eta (n - 1, \kappa, \gamma_j (n - 1, \kappa)) [g
  (n - 1, \kappa, \beta)] + 1\}$.
  
  The previous means, for any $j \in [1, \omega)$,\\
  $\alpha \leqslant^{n - 1} \eta (n - 1, \kappa, \gamma_j (n - 1, \kappa)) [g
  (n - 1, \kappa, \alpha)] + 1 = \gamma_j (n - 1, \kappa) [g (n - 1, \kappa,
  \alpha)]) + 1 =$
  
  \ \ \ \ $= \gamma_j (n - 1, \alpha) + 1$. But, by previous proposition
  \ref{most_important_sequence_in_(e,e(+^i))}, $(\gamma_j (n - 1, \alpha))_{j
  \in [1, \omega)}$ is cofinal in $\alpha (+^{n - 1})$; therefore, by
  $\leqslant^{n - 1}$-continuity follows $\alpha \leqslant^n \alpha (+^{n -
  1})$. Hence $\alpha \in \tmop{Class} (n)$.
\end{proof}

\begin{proposition}
  \label{Class(n)_indeed_k-club}Let $\kappa$ be an uncountable regular
  ordinal. Then $\tmop{Class} (n)$ is club in $\kappa$.
\end{proposition}

\begin{proof}
  We already know that $\tmop{Class} (n)$ is closed in $\kappa$. So we only
  need to show that $\tmop{Class} (n)$ is unbounded in $\kappa$.

  Let $\beta \in \kappa$.
  
  Since we know $\tmop{Class} (n - 1)$ is club in $\kappa$, take $r, r (+^{n -
  1}) \in \tmop{Class} (n - 1) \cap \kappa \neq \emptyset$.

  Consider $M^{n - 1} (r, \kappa) \assign \{q \in [\kappa, \kappa (+^{n - 1}))
  | T (n - 1, \kappa, q) \cap \kappa \subset r\}$.

  Consider $R : [r, r ( +^{n - 1})) \longrightarrow R [r ( +^{n - 1})] \subset
  [\kappa, \kappa ( +^{n - 1}))$, $R (t) \assign t [g (n - 1, r, \kappa)]$.
  Then $R$ is a bijection. We assure that $R [ {\color{magenta} [ r, r ( +^{n
  - 1}))}] = M^{n - 1} (r, \kappa)$. {\tmstrong{(a)}}

  To show $R [ {\color{magenta} [ r, r ( +^{n - 1}))}] \subset M^{n - 1} (r,
  \kappa)$. \ \ \ \ \ \ \ {\tmstrong{(a1)}}
  
  Take $t \in [r, r ( +^{n - 1}))$. Then by (2.3.1) of GenThmIH, $\tmop{Ep}
  (t) \in \tmop{Dom} (g (n - 1, r, \kappa))$ and then, by (2.2.4) of GenThmIH,
  $T (n - 1, \kappa, t [g (n - 1, r, \kappa)]) \cap \kappa = T (n - 1, r, t)
  \cap r \subset r$. Moreover, it is clear also from GenThmIH that $t [g (n -
  1, r, \kappa)] \in [\kappa, \kappa ( +^{n - 1}))$. This shows that $R (t) =
  t [g (n - 1, r, \kappa)] \in M^{n - 1} (r, t)$, and since this was done for
  $t \in r (+^{n - 1}$) arbitrary, then (a1) holds.

  To show $R [ {\color{magenta} [ r, r ( +^{n - 1}))}] \supset M^{n - 1} (r,
  \kappa)$. \ \ \ \ \ \ \ {\tmstrong{(a2)}}
  
  Let $s \in M^{n - 1} (r, \kappa)$. By (2.2.3) of GenThmIH we have that \\
  $M^{n - 1} (r, \kappa) = \{t \in [\kappa, \kappa (+^{n - 1})) | \tmop{Ep}
  (t) \subset \tmop{Dom} g (n - 1, \kappa, r)\}$. Therefore, easily from
  GenThmIH we get that $s [g (n - 1, \kappa, r)] \in [r, r ( +^{n
  - 1}))$. But then \\
  $R (s [g (n - 1, \kappa, r)]) = s [g (n - 1, \kappa, r)] [g (n - 1, r,
  \kappa)]\underset{\text{by (2.3.2) of GenThmIH}}{=}s$. This shows that \\
  $s \in R [ {\color{magenta} [ r, r ( +^{n - 1}))}]$, and since this was done
  for arbitrary $s \in M^{n - 1} (r, \kappa)$, then (a2) holds.

  (a1) and (a2) show (a).

  By (a) and (2.3.2) of GenThmIH, the function $H \assign R^{- 1} : M^{n - 1}
  (r, \kappa) \longrightarrow [r, r ( +^{n - 1}))$, \\
  $H (s) \assign s [g (n - 1, \kappa, r)]$ is a bijection. \ \ \ \ \ \ \
  {\tmstrong{(b)}}

  On the other hand, since $r \in \tmop{Class} (n - 1) \subset \mathbbm{E}
  \subset [\omega, \infty)$ (because $n - 1 \geqslant 1$), then there exists
  $\delta \in \tmop{OR}$ such that $\aleph_{\delta} = |r|$. Then
  $\aleph_{\delta} \leqslant r < \aleph_{\delta + 1} \leqslant \kappa$. But
  $\aleph_{\delta + 1}$ is a regular uncountable ordinal (because it is a
  successor cardinal), and then, by (0) of GenThmIH, $\tmop{Class} (n - 1)$ is
  club in $\aleph_{\delta + 1}$. Hence, $\aleph_{\delta} \leqslant r < r (+^{n
  - 1}) < \aleph_{\delta + 1} \leqslant \kappa$ and subsequently \\
  $| [r, r (+^{n - 1})) | \leqslant |r ( +^{n - 1}) | = |r| < \kappa$. \ \ \ \
  \ \ \ {\tmstrong{(c)}}

  Finally, from (c), (b), proposition \ref{A^n-1(t)_club_in_kapa} and
  {\cite{GarciaCornejo1}} proposition \ref{Intersection_club_classes} follows
  that the set\\
  $\bigcap_{s \in M^{n - 1} (r, \kappa)} A^{n - 1} (s)$ is club
  in $\kappa$. So there exists $\gamma \in \bigcap_{s \in M^{n - 1} (r,
  \kappa)} A^{n - 1} (s)$, with $\gamma > \beta$. But by previous proposition
  \ref{alpha_in_Class(n)_iff_alpha_in_certain_intersection_A^n-1(s)}, $\gamma
  \in \tmop{Class} (n)$. Since the previous was done for an arbitrary $\beta
  \in \kappa$, then we have shown that $\tmop{Class} (n)$ is unbounded in
  $\kappa$.
\end{proof}

{\nocite{Bachmann}} \ {\nocite{Bridge}} \ {\nocite{Buchholz1}} \
{\nocite{Buchholz2}} \ {\nocite{Buchholz3}} \ {\nocite{Buchholz4}} \
{\nocite{BuchholzSch{"u}tte1}} \ {\nocite{Carlson1}} \ {\nocite{Carlson2}} \
{\nocite{Pohlers1}} \ {\nocite{Pohlers2}} \ {\nocite{Rathjen1}} \
{\nocite{Sch{"u}tte2}} \ {\nocite{Sch{"u}tteSimpson}} \
{\nocite{Schwichtenberg1}} \ {\nocite{Setzer}} \ {\nocite{Wilken1}} \
{\nocite{Wilken2}} \ {\nocite{Wilken3}} \ {\nocite{GarciaCornejo0}} \
{\nocite{GarciaCornejo1}}

\end{document}